\documentclass[11pt,leqno]{amsart}
\usepackage{amsmath,amsfonts,latexsym,graphicx,amssymb,url, color,cite}
\usepackage{hyperref}
\usepackage[margin=0.92in]{geometry} 
\newcommand{\average}{{\mathchoice {\kern1ex\vcenter{\hrule
height.4pt width 6pt depth0pt} \kern-9.7pt}
{\kern1ex\vcenter{\hrule height.4pt width 4.3pt depth0pt}
\kern-7pt} {} {} }}

\newcommand{\bH}{\mathrm{\bf H}}

\makeatletter
\def\l@subsection{\@tocline{2}{0pt}{2.5pc}{5pc}{}}
\makeatother

\begin{document}

\newcommand{\dist}{\text{dist}} 
\newcommand{\abs}[1]{\left\vert#1\right\vert}
\newcommand{\diam}{\text{diam}}
\newcommand{\trace}{\text{trace}}
\newcommand{\R}{{\mathbb R}} 
\newcommand{\C}{{\mathbb C}}
\newcommand{\Z}{{\mathbb Z}}
\newcommand{\N}{{\mathbb N}}
\newcommand{\gradg}{\nabla_{\G}}
\newcommand{\e}{\varepsilon} 
\newcommand{\p}{\partial}

\newcommand{\calW}{{\mathrm W}}

\newcommand{\Rn}{\mathbb R^n}
\newcommand{\Rm}{\mathbb R^m}
\renewcommand{\L}[1]{\mathcal L^{#1}}
\newcommand{\G}{\mathbb G}
\newcommand{\calL}{\mathcal L}
\newcommand{\U}{\mathcal U}
\newcommand{\M}{\mathcal M}
\newcommand{\eps}{\epsilon}
\newcommand{\BVG}{BV_{\G}(\Omega)}
\newcommand{\no}{\noindent}
\newcommand{\ro}{\varrho}
\newcommand{\rn}[1]{{\mathbb R}^{#1}}
\newcommand{\res}{\mathop{\hbox{\vrule height 7pt width .5pt depth 0pt
\vrule height .5pt width 6pt depth 0pt}}\nolimits}
\newcommand{\hhd}[1]{{\mathcal H}_d^{#1}} \newcommand{\hsd}[1]{{\mathcal
S}_d^{#1}} \renewcommand{\H}{\mathbb H}
\newcommand{\BVGL}{BV_{\G,{\rm loc}}}
\newcommand{\GH}{H\G}
\renewcommand{\diam}{\mbox{diam}\,}
\renewcommand{\div}{\mbox{div}\,}
\newcommand{\divg}{\mathrm{div}_{\G}\,}
\newcommand{\norm}[1]{\|{#1}\|_{\infty}} \newcommand{\modul}[1]{|{#1}|}
\newcommand{\per}[2]{|\partial {#1}|_{\G}({#2})}
\newcommand{\Per}[1]{|\partial {#1}|_{\G}} \newcommand{\scal}[3]{\langle
{#1} , {#2}\rangle_{#3}} \newcommand{\Scal}[2]{\langle {#1} ,
{#2}\rangle}
\newcommand{\fron}{\partial^{*}_{\G}}
\newcommand{\hs}[2]{S^+_{\H}({#1},{#2})} \newcommand{\test}{\mathbf
C^1_0(\G,\GH)} \newcommand{\Test}[1]{\mathbf C^1_0({#1},\GH)}
\newcommand{\card}{\mbox{card}}
\newcommand{\bom}{\bar{\gamma}}
\newcommand{\shpiu}{S^+_{\G}}
\newcommand{\shmeno}{S^-_{\G}}
\newcommand{\CG}{\mathbf C^1_{\G}}
\newcommand{\CH}{\mathbf C^1_{\H}}
\newcommand{\di}{\mathrm{div}_{X}\,}
\newcommand{\norma}{\vert\!\vert}
\newenvironment{myindentpar}[1]%
{\begin{list}{}%
         {\setlength{\leftmargin}{#1}}%
         \item[]%
}
{\end{list}}

\renewcommand{\baselinestretch}{1.2}

\theoremstyle{plain}
\newtheorem{theorem}{Theorem}[section]
\newtheorem{corollary}[theorem]{Corollary}
\newtheorem{lemma}[theorem]{Lemma}
\newtheorem{proposition}[theorem]{Proposition}
\newtheorem{definition}[theorem]{Definition}
\newtheorem{remark}[theorem]{Remark}
\providecommand{\bysame}{\makebox[3em]{\hrulefill}\thinspace}

\renewcommand{\theequation}{\thesection.\arabic{equation}}

\title[Global $W^{1,p}$ estimates for 
linearized Monge--Amp\`ere ]{ Global $W^{1,p}$ estimates for solutions to the
linearized Monge--Amp\`ere equations}
\date{}
\author{Nam Q. Le}
\address{Department of Mathematics, Indiana University, Bloomington, IN 47405, USA}
\email{nqle@indiana.edu}
\thanks{The research of the first author was supported in part by NSF grant DMS-1500400.}
\author{Truyen Nguyen}
\address{Department of Mathematics, The University of Akron, Akron, OH 44325, USA}
\email{tnguyen@uakron.edu}
\thanks{The research of the second author was supported in part
by a grant  from the Simons Foundation (\# 318995)}

\maketitle
\begin{abstract}
 In this paper, we
 investigate regularity for solutions to the linearized Monge-Amp\`ere equations when the nonhomogeneous term has low integrability. We establish global $W^{1,p}$ estimates for all $p<\frac{nq}{n-q}$ for solutions to the  equations with right hand side in $L^q$ 
where $n/2<q\leq n$. These estimates hold under 
natural assumptions on the domain, Monge-Amp\`ere measures and boundary data. Our estimates are affine invariant analogues
of the global $W^{1,p}$ estimates of N. Winter for fully nonlinear, uniformly elliptic equations.  

\end{abstract}


\setcounter{equation}{0}
\section{Introduction and statement of the main result}\label{intro}
This paper is a sequel to \cite{LN2} and is concerned with
global $L^p$ estimates for the  derivatives of solutions to the linearized Monge-Amp\`ere equations. 
Let $\Omega\subset \R^n$ ($n\geq 2$) be a bounded convex domain and  $\phi$ be a locally uniformly convex function on  $\Omega$. The linearized  Monge-Amp\`ere equation corresponding to $\phi$ is 
\begin{equation}\label{LMA-eq}
\mathcal{L}_{\phi} u:=- \sum_{i, j=1}^{n} \Phi^{ij} u_{ij}= f\quad \mbox{in}\quad \Omega,
\end{equation}
where $$\Phi=\big(\Phi^{ij}\big)_{1\leq i, j\leq n} := (\det D^2 \phi )~ (D^2\phi)^{-1}
$$ is the cofactor matrix of the Hessian matrix $\displaystyle D^2\phi= (\phi_{ij})_{1\leq i, j\leq n}$. The operator $\mathcal{L}_\phi$ appears in
several contexts including affine differential geometry \cite{TW00,TW05,TW08,TW082}, complex geometry \cite{D05}, and fluid mechanics \cite{Br, CNP91, Loe}. Because $\Phi$ is divergence-free, that is, $\displaystyle \sum_{i=1}^n \p_i \Phi^{ij}=0$ for all $j$, we can also write $\mathcal{L}_{\phi}$ as a divergence form operator: $$\mathcal{L}_{\phi} u=- \sum_{i, j=1}^{n} \partial_i (\Phi^{ij} u_{j}).$$ We note that
the Monge-Amp\`ere equation can be viewed as a linearized Monge-Amp\`ere equation because of the identity
\begin{equation}
 \label{MA1}
 \mathcal{L}_{\phi}\phi= - n\det D^2\phi.
\end{equation}

Caffarelli 
and Guti\'errez initiated the study of the linearized Monge-Amp\`ere equations in the fundamental paper \cite{CG97}. 
There they developed an interior Harnack inequality theory for nonnegative solutions of the homogeneous
equation $\mathcal{L}_{\phi} u=0$ in terms of the pinching of the Hessian determinant
\begin{equation}\lambda\leq \det D^{2} \phi\leq \Lambda.
 \label{pinch1}
\end{equation}
This theory is an affine invariant version of the classical Harnack inequality for linear, uniformly elliptic equations with measurable coefficients.

In applications such as in the contexts mentioned above, one usually encounters the linearized Monge-Amp\`ere equations with 
the Monge-Amp\`ere measure $\det D^2\phi$ satisfying \eqref{pinch1}.
 As far as Sobolev estimates for solutions are concerned, as elucidated in \cite{GN2}, one requires
the additional assumption 
that $\det D^2\phi$ is continuous. Let us  recall that
for this latter case, $D^2\phi$ belongs to $L^p$ for all $p<\infty$ by Caffarelli's $W^{2,p}$ estimates 
\cite{C} but $D^2\phi$ is not
bounded in view of Wang's counterexamples \cite{W95}.
Notice that since $\Phi$ is positive semi-definite, $\mathcal{L}_{\phi}$ is a linear elliptic partial differential operator, 
possibly both degenerate and singular. Despite these, we still have similar regularity results, both in the interior and at the boundary, as in the classical theory 
for linear, uniformly elliptic equations such as
Harnack inequality,  H\"older, $C^{1,\alpha}$ and $W^{2,p}$ estimates; see \cite{CG97, GN1,GN2,Le13,LN2,LS1, LS2}. In \cite{LN2}, we established global $W^{2, p}$ estimates for \eqref{LMA-eq}
when the right hand side $f\in L^q(\Omega)$ for $q>\max{\{n,p\}}$ and the Monge-Amp\`ere measure $\det D^2\phi$ is continuous. Given this, one might wonder whether similar estimates hold when $f$
is less integrable. 

Due to the hidden nonlinear character of the linearized Monge-Amp\`ere equations as revealed in \eqref{MA1}, when the right hand 
side $f$ of (\ref{LMA-eq}) belongs to $L^q(\Omega)$ where $q<n$, we do not in general expect to get $W^{2,q}$ estimates for the
solutions $u$, by an example of Caffarelli \cite{C89} concerning solutions of fully nonlinear, uniformly elliptic equations. 
However, in \cite{Sw97}, \'Swiech obtained surprising $W^{1,p}$-interior estimates, for all $p<\frac{nq}{n-q}$, for a large class of fully nonlinear, uniformly elliptic equations 
\begin{equation}F(x, u(x), Du(x), D^2u(x))= f(x)
 \label{Seq}
\end{equation}
 with 
$L^q$ right hand side where $n-\e_0<q\leq n$ and $\e_0$ depends on the ellipticity constants of the equations. 
This result is almost sharp in view of the Sobolev embedding $W^{2, q}\hookrightarrow W^{1, \frac{nq}{n-q}}$. It is worth mentioning 
that \'Swiech's $W^{1,p}$ estimates in the special case
of fully nonlinear, uniformly elliptic equations of the form
\begin{equation}
 \label{Feq1}
 F(D^2 u) =f
\end{equation}
follow from Escauriaza's $W^{2, q}$ estimates \cite{Es93} for solutions of \eqref{Feq1} when $f\in L^q$ with $n-\e_0<q\leq n$. \'Swiech's $W^{1,p}$-interior estimates
were later extended up to the boundary by Winter \cite{Wi}.

In view of the aforementioned H\"older, $C^{1,\alpha}$ and $W^{2,p}$ estimates for solutions of \eqref{LMA-eq}, we
might expect $W^{1,  p}$
estimates ($p<\frac{nq}{n-q}$) for the linearized Monge-Amp\`ere equation \eqref{LMA-eq} and the main purpose of this paper is to confirm this expectation for a large range
of $q$: $n/2<q\leq n$. 
Despite the degeneracy and singularity of \eqref{LMA-eq}, that is there are no controls on the ellipticity constants,
the integrability range allowed for the right hand side of (\ref{LMA-eq}) in our main result is remarkably larger than the integrability range allowed for the right
hand side of the nonlinear, uniformly elliptic equations in the above mentioned papers of Escauriaza's, \'Swiech's and Winter's.
\subsection{The main result}
Our main result establishes global $W^{1, p}$ estimates ($p<\frac{nq}{n-q}$) for solutions to  equation \eqref{LMA-eq} 
with $L^{q}$ ($n/2<q\leq n$)
right hand side and $C^{1,\gamma}$ boundary values under natural assumptions on the domain, boundary data, and the
Monge-Amp\`ere measure. Precisely, we obtain:
\begin{theorem}[Global $W^{1,p}$ estimates]
 \label{global-reg}
Assume that there exists a small constant $\rho>0$ such that
$\Omega \subset B_{1/\rho}(0)$ and for each $y\in\p\Omega$ there is a ball $B_{\rho}(z)\subset \Omega$ that is tangent to $\p 
\Omega$ at $y$. Let  $\phi \in C^{0,1}(\overline 
\Omega) 
\cap 
C^2(\Omega)$  be a convex function satisfying
$$ \det D^2 \phi =g \, \mbox{ in }\, \Omega\quad \quad \mbox{with} \quad \lambda \le g \le \Lambda.$$ Assume further that 
on $\p \Omega$, $\phi$ 
separates quadratically from its 
tangent planes,  namely
\begin{equation}
\label{eq:sepa}
 \rho\abs{x-x_{0}}^2 \leq \phi(x)- \phi(x_{0})-D \phi(x_{0}) \cdot (x- x_{0})
 \leq \rho^{-1}\abs{x-x_{0}}^2, ~\text{for all}~ x, x_{0}\in\p\Omega.
\end{equation}
Let $u: \overline{\Omega}\rightarrow \R$ be a continuous function that 
solves the linearized Monge-Amp\`ere equation 
\begin{equation*}
 \left\{
 \begin{alignedat}{2}
   \Phi^{ij}u_{ij} ~& = f ~&&\text{in} ~ \Omega, \\\
u &= \varphi ~&&\text{on}~\p \Omega,
 \end{alignedat} 
  \right.
\end{equation*}
where $\varphi$ is a $C^{1,\gamma}$ function defined on $\p\Omega$ $(0<\gamma\leq 1)$ and $f\in L^{q}(\Omega)$ with $n/2<q\leq n$. 
Assume in addition that $g\in C(\overline{\Omega})$.
Then for any $1\leq p<\frac{nq}{n-q}$, we have the following global $W^{1,p}$ estimates
\begin{equation*}
 \|u\|_{W^{1, p}(\Omega)} \leq K\big( \|\varphi\|_{C^{1,\gamma}(\p \Omega)}+ \|f\|_{L^q(\Omega)}\big), 
\end{equation*}
where $K$ is a constant depending only on 
$n, \rho, \gamma, \lambda, \Lambda, p, q$ and the modulus of continuity of $g$.
\end{theorem}
We note from \cite[Proposition 3.2]{S2} that the quadratic separation \eqref{eq:sepa} holds 
for solutions to 
the Monge-Amp\`ere equations with the right hand side bounded away from $0$ and $\infty$ on uniformly convex domains and $C^3$ boundary data. Furthermore, 
Theorem~\ref{global-reg} complements Savin and the first author's global $C^{1,\alpha}$ estimates \cite{LS2} for 
equation \eqref{LMA-eq} when the right hand side $f$ is in $L^q$ $(q>n)$. This result is an affine invariant version of Winter's global $W^{1,p}$ estimates for fully nonlinear, uniformly
elliptic equations \cite{Wi}.

Let us  say briefly about 
the integrability range allowed for the right
hand side of  \eqref{LMA-eq}. Notice that in \cite{Sw97}, the exponent $q$ was required to be close to $n$ with the closeness depends on the ellipticity constants. Moreover, the 
proof of these $W^{1,p}$ estimates for   equation \eqref{Seq}  is rooted in  a deep integrability bound of Fabes and Stroock \cite{FS} for the Green's function of linear, uniformly elliptic operators with measurable coefficients. In a recent paper \cite{Le15}, the first author establishes the same global integrability of the Green's function for 
the linearized Monge-Amp\`ere operator as the Green's function of the Laplace operator which
corresponds to $\phi(x)=|x|^2/2$ (see also \cite{GW, L, TiW08} for previous related interior results). Namely, under the pinching condition \eqref{pinch1} and
natural boundary data, the Green's function of
$\mathcal{L}_{\phi}$ is globally $L^p$-integrable for all $p<\frac{n}{n-2}$.
Thus, as a degenerate and singular non-divergence form operator, $\mathcal{L}_\phi$ has the Green's function with global $L^p-$integrability
higher than that of a typical uniformly elliptic operator in non-divergence form as established 
in  \cite[Corollary~2.4]{FS}. This is the reason why we are able
to prove Theorem~\ref{global-reg} for a large range of $q$: $ n/2<q\leq n$.

Our strategy to proving $W^{1,p}$ estimates for solutions of (\ref{LMA-eq}) follows Caffarelli's perturbation arguments \cite{C89, CC} (see also Wang \cite{Wang}) and local boundedness and maximum principles. 
Even in the ideal case where $\phi (x)= |x|^2/2$ and (\ref{LMA-eq}) becomes the Poisson's equation $\mathcal{L}_{\phi} u= - \Delta u=f$, we do not have local boundedness for solutions when
$f$ is not $L^{n/2}$ integrable. Thus the range $n/2<q\leq n$ is almost optimal for our approach. However, our method  does not give any information for the case $q\leq n/2$. 

To prove Theorem~\ref{global-reg}, we first establish new pointwise $C^{1,\alpha}$ estimates in the interior and at the boundary 
for the linearized Monge-Amp\`ere equation \eqref{LMA-eq} with $L^{q}(\Omega)$ $(q>n/2)$ right hand side. These estimates, respectively, extend previous results of 
Guti\'errez and the second author \cite{GN1} and of Savin and the first author \cite{LS2} where the cases $q>n$ were treated. Then, we combine these pointwise estimates with the 
strong type inequality for the maximal function $\mathcal{M}$ with respect to sections of $\phi$ \cite[Theorem 2.7]{LN1} to get the desired global $W^{1,p}$ estimates.
We next indicate some more details on the proof of Theorem~\ref{global-reg} after introducing several notations.

Throughout, a convex domain $\Omega$ is called {\it normalized} if $B_1(0)\subset\Omega\subset B_n(0)$.  Also, the section of  a convex function $\phi\in C^1(\overline{\Omega})$ at $x\in\overline\Omega$ with height $h$ is defined by
\begin{equation*}\label{def:section}
S_\phi(x,h) = \Big\{y\in \overline{\Omega}: \, \phi(y)< \phi(x) + D
\phi(x)\cdot (y-x) + h\Big\}.
\end{equation*}
For fixed $\alpha\in (0, 1)$ and $r_0>0$, we denote for $z\in\overline{\Omega}$
 the following quantities $N_{\phi, f, q, r}(z)$ and $N_{\phi, f, q}(z)$:
 \begin{equation}N_{\phi, f, q, r}(z) := r^{\frac{1-\alpha}{2}}\Big(\frac{1}{|S_\phi(z, r)|}\int_{S_\phi(z, r)}{|f|^q ~dx}\Big)^{\frac{1}{q}} ~\text{for}~r>0,
\label{Nr1def}
\end{equation}
 and
\begin{equation}N_{\phi, f, q}(z) :=\sup_{ r\leq r_0}  N_{\phi, f, q, r}(z) =  \sup_{ r\leq r_0}  r^{\frac{1-\alpha}{2}}\Big(\frac{1}{|S_\phi(z, r)|}\int_{S_\phi(z, r)}{|f|^q ~dx}\Big)^{\frac{1}{q}}
 \label{Ndef}.
\end{equation}
We will use the letters $c, c_1, C, C_1, C', C^{\ast}, \theta_{\ast}, \bar\theta, ...$, etc, to denote generic constants depending only on 
the structural constants $n, q, \rho, \gamma, \lambda, \Lambda$ that may 
change from line to line. They are called {\it universal constants}.

 We can assume that all functions $\phi$, $u$ in this paper are smooth. However, our estimates do not depend on the assumed smoothness but only on the given structural constants.

The main points  of the proof of Theorem \ref{global-reg} are as follows. 
By the global maximum principle using the optimal integrability of the Green's function of the  operator $\mathcal{L}_{\phi}$, we have
 \begin{equation}
  \label{uu-max}
  \|u\|_{L^{\infty}(\Omega)} \leq C \left( \|\varphi\|_{L^{\infty}(\Omega)}
  +\|f\|_{L^{q}(\Omega)} \right).
 \end{equation}
Let $q'\in (n/2, q)$. By applying  the foregoing pointwise $C^{1,\alpha}$ estimates in the interior and at the boundary
for \eqref{LMA-eq}, we obtain the following gradient bound:
\begin{equation}
\label{eq:step1}
|Du(y)|\leq C\Big(\|u\|_{L^{\infty}(\Omega)} + N_{\phi, f, q'}(y)\Big)\quad \forall y\in\Omega.\end{equation}
Note that $N(y) :=N_{\phi, f, q'}(y)$ can be $\infty$. However, using volume estimates for sections of $\phi$, we find that for $p\geq q>q'$
$$\|N\|_{L^p(\Omega)}\leq C \sup_{r\leq r_0} \Big\{ r^{\frac12\big[(1-\alpha) -\frac{ n}{q} +\frac{n}{p}\big]}  \Big\}
\left( \int_{\Omega} \M(f^{q'})(y)^{\frac{q}{q'}}  dy\right)^{\frac{1}{p}}
\|f\|_{L^{q}(\Omega)}^{\frac{p-q}{p}}.$$
We then employ the strong-type $\frac{q}{q'}-\frac{q}{q'}$ inequality for the maximal function $\M(f^{q'})$ with respect to sections of $\phi$ and, since $p<\frac{nq}{n-q}$, we can 
choose $0<\alpha< 1-\frac{n}{q}+ \frac{n}{p}$ to conclude that
\begin{equation}
\label{eq:step2}\|N\|_{L^p(\Omega)} \leq C \sup_{r\leq r_0} \Big\{ r^{\frac12\big[(1-\alpha) -\frac{ n}{q} +\frac{n}{p}\big]}  \Big\} \|f\|_{L^q(\Omega)}^{\frac{q}{p}}
\|f\|_{L^{q}(\Omega)}^{\frac{p-q}{p}}\leq C  \|f\|_{L^q(\Omega)}.
\end{equation}
By combining \eqref{uu-max}--\eqref{eq:step2}, we obtain the global $W^{1,p}$ estimate in Theorem~\ref{global-reg}.

\subsection{Key estimates} As mentioned above, the new key estimates in the proof of Theorem \ref{global-reg} are 
pointwise $C^{1,\alpha}$ estimates in the interior and at the boundary 
for solutions to the linearized Monge-Amp\`ere equation \eqref{LMA-eq} with $L^q$ right hand side where $q>n/2$. 

We first state pointwise $C^{1,\alpha}$ estimates in the interior. 
 
\begin{theorem} [Pointwise $C^{1,\alpha}$ estimates in the interior]
\label{HolderDv}
Assume that $q>n/2$,  $0\leq \alpha'<\alpha<1$, and   $r_0>0$. There exists
$\theta=\theta(n,q,\alpha, \alpha', r_0)>0$  such that if
$\Omega$ is a {\it normalized} convex domain, $\phi\in C(\overline\Omega)$ is a convex solution of
$$1-\theta \leq \det D^2 \phi \leq 1+\theta \text{ in } \Omega,~ \text{and~}
\phi =0\text{ on } \partial \Omega, $$
then any  solution $u\in W^{2,n}_{loc}(\Omega)$ of $\Phi^{ij} u_{i j} = f$ in $\Omega$ where $f\in L^q(\Omega)$ satisfies the following pointwise $C^{1,\alpha'}$ estimate at 
 the minimum point $\bar z$  of $\phi$:
\begin{align*}
& r^{-(1 + \alpha')}\|u-l\|_{L^\infty(B_r(\bar z))} + |l(\bar z)| +\|D l\| \leq C \Big[ \|u\|_{L^\infty(\Omega)} + N_{\phi, f, q}(\bar z)\Big]\qquad\text{ for all } r\leq \mu^*,
\end{align*}
for some affine function $l$, where 
 $C$,  $\mu^*$ are positive constants depending only on $n$, $q$, $\alpha$, $\alpha'$ and $r_0$.
\end{theorem}
Note that, in the above theorem, $\det D^2\phi$ is only required to be close to a positive constant, but no continuity of $\det D^2\phi$ 
 is needed. Theorem \ref{HolderDv} extends a previous result of Guti\'errez and the second author \cite[Theorem 4.5]{GN1} from the case $q=n$ to all $q$ satisfying $n/2<q\leq n$.

 The interior $W^{1,p}$ estimates for (\ref{LMA-eq}) then follow.
\begin{theorem}[Interior $W^{1,p}$ estimates]
\label{Interior-Du}
Let $\Omega$ be a {\it normalized} convex domain and  $\phi\in C(\overline\Omega)$ be a convex solution to $\det D^2\phi =g$ in 
$\Omega$ and $\phi =0$ on $\partial\Omega$, where  $g\in C(\Omega)$ satisfying $\lambda \leq g(x)\leq \Lambda$ in $\Omega$. 
Suppose that  $u\in W^{2,n}_{loc}(\Omega)$ is a  solution of $\Phi^{ij} u_{i j} = f$ in $\Omega$ with $f\in L^q(\Omega)$ where
$n/2<q\leq n$. Then for any $\Omega'\Subset \Omega$ and any $p<\frac{nq }{n-q}$,  we have
  \begin{equation}\label{interiorW1p-est}
 \|Du\|_{L^p(\Omega')} \leq 
C \Big(  \|u\|_{L^\infty(\Omega)}
+ \|f\|_{L^q(\Omega)} \Big),
 \end{equation}
where $C>0$ depends only on $n$, $p$, $q$, $\lambda$, $\Lambda$, $\dist(\Omega', \partial   \Omega)$ and the modulus of continuity of $g$.
\end{theorem}

We next state
pointwise $C^{1,\alpha}$ estimates at the boundary for solutions of \eqref{LMA-eq} with $L^{q}$ right hand side where $q>n/2$ and 
$C^{1,\gamma}$ boundary data under the local assumptions \eqref{om_ass}--\eqref{eq_u1} introduced below.
These estimates generalize previous results of Savin and the first author in \cite{LS1, LS2} where the cases $q=\infty$ and $q>n$, respectively, were treated.

Let $\Omega\subset \R^{n}$ be a bounded convex set with
\begin{equation}\label{om_ass}
B_\rho(\rho e_n) \subset \, \Omega \, \subset \{x_n \geq 0\} \cap B_{\frac 1\rho} (0),
\end{equation}
for some small $\rho>0$ where we denote $e_n:= (0,\dots, 0, 1)\in \R^n$. Assume that 
\begin{equation}
\text{for each~} y\in\p\Omega\cap B_\rho(0), ~\text{there is a ball~} B_{\rho}(z)\subset \Omega \text{ that is tangent to } \p 
\Omega \text{ at } y.
\label{tang-int}
\end{equation}
Let 
 $\phi \in C^{0,1}(\overline 
\Omega) 
\cap 
C^2(\Omega)$  be a convex function satisfying
\begin{equation}\label{eq_u}
 0 <\lambda \leq \det D^2\phi \leq \Lambda \quad \text{in $\Omega$}.
\end{equation}
We assume that on $\p \Omega\cap B_\rho(0)$, 
$\phi$ separates quadratically from its tangent planes on $\p \Omega$. 
Precisely we assume that if $x_0 \in 
\p \Omega \cap B_\rho(0)$ then
\begin{equation}
 \rho\abs{x-x_{0}}^2 \leq \phi(x)- \phi(x_{0})-D\phi(x_{0}) \cdot (x- x_{0}) \leq 
\rho^{-1}\abs{x-x_{0}}^2\quad \text{ for all } x \in \p\Omega.
\label{eq_u1}
\end{equation}

\begin{theorem}
\label{h-bdr-gradient}
Assume that $\phi$ and $\Omega$ satisfy  assumptions 
\eqref{om_ass}--\eqref{eq_u1}. Let $u: B_{\rho}(0)\cap 
\overline{\Omega}\rightarrow \R$ be a continuous solution to 
\begin{equation*}
 \left\{
 \begin{alignedat}{2}
   \Phi^{ij}u_{ij} ~& = f ~&&\text{in} ~ B_{\rho}(0)\cap \Omega, \\\
u &= \varphi~&&\text{on}~\p \Omega \cap B_{\rho}(0),
 \end{alignedat} 
  \right.
\end{equation*} 
where $f\in L^{q}(B_{\rho}(0)\cap\Omega)$ for some $q>n/2$ and $\varphi \in C^{1,\gamma}(B_{\rho}(0)\cap\p\Omega)$.
Then there exist $\alpha\in (0, 1)$ and $\theta$ small depending only on 
$n, q, \rho, \lambda, \Lambda, \gamma$ such that for all $\bar h\leq \theta^2$, we can find $b\in\R^n$
satisfying 
\begin{multline*}\bar h^{-\frac{1+\alpha}{2}}\|u- u(0)-bx\|_{L^{\infty}(S_{\phi}(0, \bar h))} + \|b\|
 \leq C \Big[\|u\|_{L^{\infty}(B_{\rho}(0)\cap\Omega)}  +
 \|\varphi\|_{C^{1,\gamma}(B_{\rho}(0)\cap\p\Omega)} + \sup_{ \bar h\leq t\leq \theta^2} N_{\phi, f, q, 2\theta^{-1} t}(0)\Big]
 \end{multline*}
where $C$ depends only on $n, q, \rho, \lambda, \Lambda$, and $\gamma$. We can take  $\alpha\in (0, \min\{\alpha_0, \gamma\})$
where $\alpha_0$ is the exponent in the boundary H\"older gradient estimates, Theorem~\ref{LS-gradient}.
\end{theorem}

In proving global $W^{1,p}$ estimates for solutions of  \eqref{LMA-eq}, we will use new
maximum principles, in the interior and at the boundary, for the linearized Monge-Amp\`ere equation \eqref{LMA-eq} with $L^q$ right hand side where $q$ is only assumed to satisfy 
$q>n/2$. We state here a global maximum principle and refer to Lemmas \ref{lm:inter-maximum-prin} and \ref{lm:glob-maximum-prin} for the interior and boundary maximum principles
used in the paper.
\begin{lemma}[Global maximum principle]\label{lm:glob-maximum-prin2} Assume that $\Omega$ and $\phi$ satisfy the hypotheses of Theorem~\ref{global-reg} up to (\ref{eq:sepa}).
Let $f\in L^q(\Omega)$ for some $q>n/2$  and  $u\in W^{2,n}_{loc}(\Omega)\cap C(\overline{\Omega})$ satisfy 
$$
\mathcal{L}_{\phi} u\leq f\quad \mbox{almost everywhere in}\quad \Omega.
$$
Then  there exists a constant $C>0$ depending only on  $n$, $\lambda$, $\Lambda$, $\rho$ and  $q$  such that
\[
\sup_{\Omega}{u} \leq \sup_{\partial\Omega}{u^+} + C |\Omega|^{\frac{2}{n} -\frac{1}{q}}   \|f\|_{L^q(\Omega)}. 
\]
\end{lemma}
We will also use the following global strong type estimates for the maximal function $\mathcal{M}$
with respect to sections of the potential function $\phi$. 
\begin{theorem}(Strong-type $p$--$p$ estimates, \cite[Theorem 2.7]{LN1})\label{strongtype}
Assume that $\Omega$ and $\phi$ satisfy the hypotheses of Theorem~\ref{global-reg} up to (\ref{eq:sepa}).
For $f\in L^1(\Omega)$,  define
\begin{equation*}\label{maximalfunction}
\mathcal M (f)(x)=\sup_{t>0}
\dfrac{1}{|S_\phi(x,t)|}\int_{S_\phi(x,t)}|f(y)|\, dy\quad \text{ for all } x\in \Omega.
\end{equation*}
Then, for any $1<p<\infty$, there exists  $C_p>0$ depending on $p$, $\rho$, $\lambda$, $\Lambda$ and $n$ such that $$\|\mathcal{M}(f)\|_{L^{p}(\Omega)}\leq C_p\, \|f\|_{L^{p}(\Omega)}.$$
\end{theorem}

Note that our new maximum principles in Lemmas \ref{lm:inter-maximum-prin} and \ref{lm:glob-maximum-prin} allow us
to establish global H\"older continuity estimates for solutions to 
the linearized Monge-Amp\`ere equation \eqref{LMA-eq} with $L^q$ right hand side where $q$ is only assumed to satisfy 
$q>n/2$. These estimates in turn extend our previous results, \cite[Theorem~1.4]{Le13} and \cite[Theorem 4.1]{LN2}, where the cases of $L^n$ right hand side were treated.
\begin{theorem}[Global H\"older estimates ]\label{global-H}
Assume $\Omega$ and $\phi$  satisfy  \eqref{om_ass}--\eqref{eq_u1}.
Let $u \in C\big(B_{\rho}(0)\cap 
\overline{\Omega}\big) \cap W^{2,n}_{loc}(B_{\rho}(0)\cap 
\Omega)$  be a  solution to 
\begin{equation*}
 \left\{
 \begin{alignedat}{2}
   \Phi^{ij}u_{ij} ~& = f ~&&\text{ in } ~ B_{\rho}(0)\cap \Omega, \\\
u &= \varphi~&&\text{ on }~\p \Omega \cap B_{\rho}(0),
 \end{alignedat} 
  \right.
\end{equation*} 
where $\varphi\in C^{\alpha}(\partial\Omega\cap B_{\rho}(0))$ for some $\alpha\in (0,1)$ and $f\in L^q(\Omega\cap B_\rho(0))$. Then for any $q>n/2$, there exist constants $\beta, C >0 $ depending only on $\lambda, \Lambda, n, \alpha$, $q$ and $\rho$ 
such that 
$$|u(x)-u(y)|\leq C|x-y|^{\beta}\Big(
\|u\|_{L^{\infty}(\Omega\cap B_{\rho}(0))} + \|\varphi\|_{C^\alpha(\partial\Omega\cap B_{\rho}(0))}  + \|f\|_{L^q(\Omega\cap B_{\rho}(0))} \Big)~\text{for all }x, y\in \Omega\cap 
B_{\frac{\rho}{2}}(0). $$
\end{theorem}
\smallskip
The rest of the paper is organized as follows.
In Section~\ref{sec:MaximumHolder}, we establish an interior maximum principle, an interior H\"older estimate, and a comparison estimate for the linearized Monge-Amp\`ere
equations with $L^q$ right hand side. 
We prove Theorems \ref{HolderDv} and \ref{Interior-Du} in Section~\ref{sec:interiorw1p}. The proofs of Theorem~\ref{h-bdr-gradient} and Lemma~\ref{lm:glob-maximum-prin2} will 
be given in Section~\ref{sec:boundaryw1p}. 
In the final Section~\ref{sec:globalw1p}, we prove 
Theorems~\ref{global-reg} and~\ref{global-H}.

\section{Interior maximum principle and H\"older estimates}\label{sec:MaximumHolder}
In this section, we prove an interior maximum principle (Lemma~\ref{lm:inter-maximum-prin}), an interior H\"older estimate (Corollary~\ref{cor:imterior-holder}),  and a comparison estimate
(Lemma~\ref{explest}) for the linearized Monge-Amp\`ere equation with $L^q$ right hand side where $q$ is only assumed to satisfy $q>n/2$. These results will be used in
Section~\ref{sec:interiorw1p} to prove interior $W^{1,p}$ estimates.

For convenience, we introduce the following hypothesis:\\
 $(\bH)$ \ \ \ {\it $\Omega$ is a {\it normalized} convex domain  and $\phi \in C(\overline\Omega)$ is a convex function such that}
\[\lambda \leq \text{det} D^2\phi \leq \Lambda \mbox{ in } \Omega \quad \mbox{ and }\quad \phi =0 \mbox{ on }
\partial\Omega.
\]

 Given $0 <\alpha <
1$, and $\Omega$ and $\phi$ satisfying ({\bf H}),  we define the sections of $\phi$ at its minimum point $\bar z$ to be the sets
\begin{equation*}
\Omega_{\alpha}\equiv \Omega_{\alpha,\phi}:=S_\phi(\bar z, -\alpha \min_{\Omega}\phi) =\Big\{x\in \overline{\Omega}:\,  \phi(x)<(1-\alpha
)\,\min_{\Omega}\phi \Big\}.
\end{equation*}

We record here how the linearized Monge-Amp\`ere equation \eqref{LMA-eq} transforms under rescaling. If $Tx= Ax + z$ is an affine transformation
where $A$ is an $n\times n$ invertible matrix and $z\in\R^n$, and
\begin{equation*}
 \tilde{\phi}(x)= \frac{1}{a}\phi(Tx), \quad \quad \tilde{u}(x) = \frac{1}{b}u(Tx), 
\end{equation*}
then from \eqref{LMA-eq}, we find
\begin{equation}
 \label{trans1}
 \mathcal{L}_{\tilde\phi} \tilde u(x) = \frac{1}{a^{n-1}b} (\det A)^{2} f(Tx).
\end{equation}
Indeed, we can compute
\begin{eqnarray*}
 D^{2}\tilde{\phi}= \frac{1}{a}A^{t} D^{2} \phi A, \quad D^{2}\tilde{u}= 
\frac{1}{b}A^{t} D^{2} u A,
\end{eqnarray*}
and the cofactor matrix $\tilde \Phi=(\det  D^{2}\tilde{\phi}) ( D^{2}\tilde{\phi})^{-1} $ of $D^2\tilde \phi$ is
\begin{eqnarray*}
 \tilde{\Phi} 
=\frac{1}{a^{n-1}}(\det A)^2 (\det D^{2} \phi) \, A^{-1} (D^{2} \phi)^{-1} 
(A^{-1})^{t}=\frac{1}{a^{n-1}}(\det A)^2 A^{-1} \Phi (A^{-1})^{t}.
\end{eqnarray*}
Thus (\ref{trans1}) easily follows from
$$\mathcal{L}_{\tilde\phi} \tilde u(x) = -\text{trace} (\tilde \Phi D^2 \tilde u)= -\frac{1}{a^{n-1}b} (\det A)^{2}\text{trace} ( \Phi D^2 u (Tx))= 
\frac{1}{a^{n-1}b} (\det A)^{2} f(Tx).$$
\subsection{Interior estimates}
\begin{lemma}[Interior maximum principle]\label{lm:inter-maximum-prin}
Assume that $\Omega$ and $\phi$ satisfy {\bf (H)}. 
 Let $V\subset \Omega$ be a subdomain,
 $f\in L^q(V)$ for some $q>n/2$,  and  $u\in W^{2,n}_{loc}(V)\cap C(\overline{V})$ satisfy 
\[
\mathcal{L}_{\phi} u\leq f\quad \mbox{almost everywhere in}\quad V.
\]
Then for any $\alpha\in (0, 1)$, there exists a constant $C>0$ depending only on $\alpha$, $n$, $\lambda$, $\Lambda$ and  $q$  such that
\[
\sup_{V\cap \Omega_\alpha}{u} \leq \sup_{\partial V}{u^+} + C |V|^{\frac{2}{n} -\frac{1}{q}}   \|f\|_{L^q(V)}. 
\]
\end{lemma}
\begin{proof}
Let $G_V(x,y)$ be the Green's function of $\mathcal{L}_{\phi}$ in $V$ with pole $y\in V$, namely $G_V(\cdot, y)$ is a positive solution of 
\begin{equation*}
\left\{\begin{array}{rl}
\mathcal{L}_{\phi} G_V( \cdot, y) &=\delta_y \qquad \mbox{ in}\quad V,\\
G_V( \cdot, y) &=0\ \ \ \ \ \ \ \mbox{on}\quad \partial V
\end{array}\right.
\end{equation*}
with $\delta_y$ denoting the Dirac measure giving unit mass to the point $y$. Define
\[
v(x) := \int_{V} G_V(x,y) f(y)\, dy \quad\mbox{for}\quad x\in V.
\]
Then $v$ is a solution of 
$$\mathcal{L}_{\phi} v =f \mbox{ in }V,~\text{and}~
v =0\mbox{ on } \partial V.$$
Since $\mathcal{L}_{\phi}(u-v)\leq 0$ in $V$, we obtain from the Aleksandrov-Bakelman-Pucci (ABP) maximum principle (see \cite[Theorem~9.1]{GiT}) that
\begin{equation}\label{comparison-principle}
u(x) \leq \sup_{\partial V}{u^+} + v(x)\quad \mbox{in} \quad V.
\end{equation}
We next estimate $v(x)$ for the case $n\geq 3$ using \cite[Lemma~3.3]{TiW08}. The case $n=2$ is treated similarly, using \cite[Theorem 1.1]{L}.
Notice that  Aleksandrov's estimate (see \cite[Theorem~1.4.2]{G01}) implies that $\dist(\Omega_\alpha, \partial\Omega)\geq c(n,\lambda,\Lambda) (1-\alpha)^n>0$. It follows 
from this and the proof of \cite[Lemma 3.3]{TiW08} that  there exists a constant $K>0$ depending on $\alpha$, $n$, $\lambda$ and $\Lambda$
such that for every $y\in V\cap \Omega_\alpha$ we have
\begin{equation}\label{dist-est-Green}
|\{x\in V: ~ G_V(x,y)>t\}|\leq K t^{-\frac{n}{n-2}}\quad \mbox{ for}\quad t>0.
\end{equation}
As the operator $\mathcal{L}_\phi$ can be written in the divergence form with symmetric coefficient, we  infer 
 from \cite[Theorem~1.3]{GW}
that $G_V(x,y)=G_V(y,x)$ for all $x,y\in V$. This together with \eqref{dist-est-Green} allows us to deduce that
for every $x\in V \cap \Omega_\alpha$, there holds
\begin{equation*}
|\{y\in V: ~ G_V(x,y)>t\}|
=|\{y\in V: ~ G_V(y,x)>t\}| \leq K t^{-\frac{n}{n-2}}\quad \mbox{ for}\quad t>0.
\end{equation*}
It follows that if $q>\frac{n}{2}$, then $q':= \frac{q}{q-1}<\frac{n}{n-2}$ and from the layer cake representation, we have
\begin{align*}
\int_{V}{G_V(x,y)^{q'}~dy}
&= q' \int_0^\infty t^{q' -1}|\{y\in V: ~ G_V(x,y)>t\}| \, dt\\
&\leq q' |V| \int_0^\epsilon t^{q' -1} \, dt + q'K  \int_\epsilon^\infty t^{q' -1-\frac{n}{n-2}} \, dt= |V| \epsilon^{q'} + C_1 \epsilon^{q'- \frac{n}{n-2}}  \text{ for all} ~\epsilon>0.
\end{align*}
By choosing $\epsilon = \big(\frac{C_1}{|V|}\big)^{\frac{n-2}{n}}$ in the above right hand side, we obtain
\begin{align*}
\sup_{x\in V \cap \Omega_\alpha}\int_{V}{G_V(x,y)^{q'}~dy}
\leq  2C_1^{\frac{n-2}{n} q'} |V|^{1- \frac{n-2}{n} q'}.
\end{align*}
We deduce from  the definition of $v$, H\"older inequality and the above estimate for $G_V$ that
\begin{align*}
|v(x)| \leq  \|G_V(x,\cdot)\|_{L^{q'}(V)}\|f\|_{L^q(V)}\leq 2C_1^{\frac{n-2}{n}} |V|^{\frac{1}{q'}- \frac{n-2}{n} } \|f\|_{L^q(V)}\quad
\text{for all }  x\in V \cap \Omega_\alpha.
\end{align*}
This estimate and  \eqref{comparison-principle} yield the conclusion of the lemma.
\end{proof}
By employing Lemma~\ref{lm:inter-maximum-prin} and the  interior Harnack inequality established in \cite{CG97} for nonnegative solutions to the homogeneous linearized Monge-Amp\`ere equations, we get:
\begin{lemma}[Harnack inequality]\label{lm:Harnackinequality}
Assume that $\Omega$ and $\phi$ satisfy {\bf (H)}.
Let $f\in L^q(\Omega)$ for some $q>n/2$  and  $u\in W^{2,n}_{loc}(\Omega)$ satisfy  $\mathcal{L}_{\phi} u= f$ almost everywhere in $\Omega$.
Then  if  $S_\phi(x, t)\Subset \Omega$ and $u\geq 0$ in $S_\phi(x, t)$, we have
\begin{equation}\label{eq:Harnack}
\sup_{S_\phi(x, \frac{t}{2})}{u} \leq C\Big( \inf_{S_\phi(x, \frac{t}{2})}{u} + |S_\phi(x,t)|^{\frac{2}{n} -\frac{1}{q}}  \, \|f\|_{L^q(S_\phi(x, t))}\Big),
\end{equation}
where $C>0$ depends only on $n$, $\lambda$, $\Lambda$ and  $q$.
\end{lemma}
\begin{proof}
For convenience, let us write $S_h$ for the section  $S_\phi(x, h)$. 
Let $u_0$ be the solution of 
$$
\mathcal{L}_{\phi} u_0 =f~\text{ in }S_t,~\text{and}~
u_0 =0~\text{ on } \partial S_t.
$$
Then $\mathcal{L}_{\phi} (u-u_0)=0$ in $S_t$ and $u-u_0\geq 0$ on $\partial S_t$. Thus we conclude from the ABP maximum  principle that $u-u_0\geq 0$ in $S_t$. Hence, we can apply the interior Harnack inequality established in \cite[Theorem~5]{CG97} to obtain \[
\sup_{S_{\frac{t}{2}}} (u- u_0) \leq C \inf_{S_{\frac{t}{2}}} (u- u_0),
\]
for some constant $C$ depending only on $n, \lambda,$ and $
\Lambda$, 
which then implies
\[
\sup_{S_{\frac{t}{2}}} u \leq C'\Big( \inf_{S_{\frac{t}{2}}} u + \sup_{S_{\frac{t}{2}}} |u_0|\Big).
\]
By normalizing the section $S_t$, $\phi$, $u_0$ and applying Lemma~\ref{lm:inter-maximum-prin} for $\alpha=1/2$, we get
\begin{equation}
\label{uomax}
\sup_{S_{\frac{t}{2}}}{|u_0|} \leq  C |S_t|^{\frac{2}{n} -\frac{1}{q}}   \|f\|_{L^q(S_t)}.
\end{equation}
Therefore, estimate \eqref{eq:Harnack} follows as desired. 

For reader's convenience, we include the details of (\ref{uomax}). By subtracting a linear function from $\phi$,
we can assume that $\phi(x)=0$ and $D\phi(x)=0$. By John's lemma, there is an affine transformation $Ty= Ay + z$ such that
\begin{equation}
 \label{normal1}
 B_1(0)\subset \tilde \Omega:= T^{-1} S_\phi(x,t) \subset B_n(0),
\end{equation}
where $A$ is an $n\times n$ invertible matrix and $z\in\R^n$.
Rescale $\phi$ and $u_0$ by
$$\tilde \phi(y)= \frac{1}{|\det A|^{2/n}} [\phi(Ty)-t],~\tilde u_0 (y) = u_0 (Ty),\quad y\in\tilde \Omega.$$
Then $\tilde \Omega$ and $\tilde \phi$ satisfy ({\bf H}). 
Moreover, by using \eqref{trans1} with $a=|\det A|^{2/n}$ and $b=1$, we find
$$\mathcal{L}_{\tilde\phi}\tilde u_0(y)= |\det A|^{2/n} f(Ty):= \tilde f(y)~\text{in}~\tilde\Omega~
\text{ with } ~\tilde u_0=0~ \text{on}~ \p\tilde\Omega.$$ Therefore, we can apply Lemma~\ref{lm:inter-maximum-prin} for $V=\tilde\Omega$ and $\alpha=1/2$ to get
\begin{equation}
 \label{estb}
 \sup_{y\in \tilde\Omega_{\frac{1}{2},\tilde\phi}} |\tilde u_0(y)| \leq C(n,\lambda,\Lambda, q) |\tilde\Omega|^{\frac{2}{n}-\frac{1}{q}} \|\tilde f\|_{L^q(\tilde \Omega)}.
\end{equation}
Since
$$\|\tilde f\|_{L^q(\tilde \Omega)} = |\det A|^{\frac{2}{n}-\frac{1}{q}} \|f\|_{L^q (S_\phi(x, t))},$$
and by (\ref{normal1}),
$$C_1^{-1}(n) |S_\phi(x, t)|\leq |\det A|\leq C_1(n) |S_\phi(x, t)|,$$
we find from (\ref{estb}) that
$$\sup_{ S_\phi(x, t)} |u_0|=\sup_{y\in \tilde\Omega_{\frac{1}{2},\tilde\phi}} |\tilde u_0(y)|  \leq 
C(n,\lambda,\Lambda, q)|S_\phi(x, t)|^{\frac{2}{n}-\frac{1}{q}}\|f\|_{L^q (S_\phi(x, t))}. $$
This proves (\ref{uomax}), completing the proof of the lemma.
\end{proof}

As a consequence of Lemma~\ref{lm:Harnackinequality}, we obtain the following  oscillation estimate:

\begin{corollary}\label{cor:oscillation-est} Assume that $\Omega$ and $\phi$ satisfy {\bf (H)}.
Let $f\in L^q(\Omega)$ for some $q>n/2$  and  $u\in W^{2,n}_{loc}(\Omega)$ satisfy  $\mathcal{L}_{\phi} u= f$ almost everywhere in $\Omega$.
Then  for any section $S_\phi(x, h)\Subset \Omega$, we have
\begin{equation*}
\text{osc}_{S_\phi(x,\rho)}{u} \leq C\big(\frac{\rho}{h}\big)^\alpha \Big[
\text{osc}_{S_\phi(x,h)}{u}  + h^{1 -\frac{n}{2 q}}  \, \|f\|_{L^q(S_\phi(x, h))}\Big]\quad \mbox{for all}\quad \rho\leq h,
\end{equation*}
where $\displaystyle \text{osc}_E u :=\sup_{E} u -\inf_{E} u$ and the constants $C,\,\alpha>0$  depend only on $n$, $\lambda$, $\Lambda$, and $q$.
\end{corollary}
\begin{proof}
Let us  write $S_t$ for the section  $S_\phi(x, t)$. Then, by \cite[Corollary 3.2.4]{G01}, 
 there exist constants $C$ and $C'$ depending only on $n,\lambda,\Lambda$ such that the volume of interior sections of $\phi$ satisfies 
 $$ C t^{n/2}\leq |S_t|\leq C' t^{n/2}~ \text{whenever}~
 S_t\Subset \Omega.$$ 
Set
\begin{align*}
m(t) := \inf_{S_t} u,\quad  M(t) :=\sup_{S_t} u,\quad \mbox{and} \quad \omega(t) := M(t) -m(t).
\end{align*}
Let $\rho\in (0, h]$ be arbitrary.
Then since  $\tilde u := u - m(\rho)$ is a nonnegative solution of $\mathcal{L}_{\phi} \tilde u= f$ in $S_\rho$,
we can apply Lemma~\ref{lm:Harnackinequality} for $\tilde u$  and the volume growth of interior sections of 
$\phi$ to obtain
\[
\frac{1}{C} \sup_{S_{\frac{\rho}{2}}}\tilde u\leq \inf_{S_{\frac{\rho}{2}}}\tilde u +\rho^{1 -\frac{n}{2 q}}  \, \|f\|_{L^q(S_\rho)}.
\]
It follows that for all $\rho\in (0, h]$, we have
\begin{align*}
\omega(\frac{\rho}{2})
= \sup_{S_{\frac{\rho}{2}}}\tilde u- \inf_{S_{\frac{\rho}{2}}}\tilde u\leq \big(1- \frac{1}{C}\big)\sup_{S_{\frac{\rho}{2}}}\tilde u  +\rho^{1 -\frac{n}{2 q}}  \, \|f\|_{L^q(S_\rho)}
\leq \big(1- \frac{1}{C}\big)\omega(\rho) +\rho^{1 -\frac{n}{2 q}}  \, \|f\|_{L^q(S_h)}.
\end{align*}
Thus, by the standard iteration we deduce that
\begin{align*}
\omega(\rho)
\leq C'\big(\frac{\rho}{h}\big)^\alpha \Big[ \omega(h) +h^{1 -\frac{n}{2 q}}  \, \|f\|_{L^q(S_h)}\Big],
\end{align*}
giving the conclusion of the corollary.
\end{proof}

Corollary~\ref{cor:oscillation-est} implies H\"older estimate. Indeed, from the arguments in \cite[pp. 456-457]{CG97}, we have
\begin{equation*}
|u(x)-u(y)| \leq C \|A\|^\beta |x-y|^\beta\Big[ \|u\|_{L^\infty(S_\phi(x_0, 2h))} + (2h)^{1-\frac{n}{2 q}}\|f\|_{L^q(S_\phi(x_0, 2h))}\Big]
\quad \text{for all } x,y\in S_\phi(x_0, h)
\end{equation*}
where $C$ is a universal constant 
and $Tx= A(x-x_0) +y_0$ is the affine transformation normalizing $S_\phi(x_0, 2\theta h)$, i.e., 
$B_1(0) \subset T\big(S_\phi(x_0, 2\theta h) \big)\subset B_n(0)$ ($\theta=\theta(n,\lambda,\Lambda)>1$ is the engulfing constant).
But when $\Omega$ is {\it normalized}, we have from \cite[Theorem 3.3.8]{G01} the inclusion
$B_{c_1 h}( x_0) \subset S_\phi (x_0, h).$
Therefore $A B_{c_1 h}(0) + y_0 \subset B_n(0)$ and hence $\|A\|\leq C h^{-1}$. Consequently,
\begin{equation*}
|u(x)-u(y)| \leq C^* h^{-\beta} |x-y|^\beta\Big[ \|u\|_{L^\infty(S_\phi(x_0, 2h))} + (2h)^{1-\frac{n}{2q}}\|f\|_{L^q(S_\phi(x_0, 2h))}\Big]
\,\, \text{for all } x,y\in S_\phi(x_0, h)
\end{equation*}
where $C^*$ is a universal constant. From this, we deduce the next result.
 \begin{corollary}[Interior H\"older estimate]\label{cor:imterior-holder}
Assume that $\Omega$ and $\phi$ satisfy $(\bH)$. Let  $f\in L^q(B_1(0))$ for some $q>n/2$ and $u\in W^{2,n}_{loc}(B_1(0))$ be a solution of 
$\mathcal{L}_{\phi} u= f$ in $B_1(0)$. Then there exist constants $\beta\in (0,1)$ and $C>0$ depending only on  $n$, $\lambda$, $\Lambda$, $q$ such that
\[
|u(x) - u(y)|\leq C |x-y|^\beta \Big( \|u\|_{L^\infty(B_1(0))} +\|f\|_{L^q(B_1(0))}\Big)\quad \text{for all } x,y\in B_{\frac12}(0).
\]
\end{corollary}

\subsection{Comparison and stability estimates}

The following lemma allows us to compare explicitly two solutions originating from two different linearized Monge-Amp\`ere equations.  
\begin{lemma}\label{explest} 
Let $U$ be a {\it normalized} convex domain. 
Assume that $\phi, w\in C(\overline U)$ are convex functions satisfying 
$\frac{1}{2} \leq \det D^2 \phi\leq \frac{3}{2}$, $\det D^2 w =1$ in $U$ and $\phi = w =0$ on $\partial U$. 
Let $\Phi= (\Phi^{ij})$ and   $\calW= (W^{ij})$  be the cofactor matrices of $D^2\phi$ and $D^2w$, respectively. Denote $U_{\alpha}= U_{\alpha, \phi}$ for $0<\alpha<1$.
Assume that $u\in W^{2,n}_{loc}(U)\cap C(\overline U)$ satisfies
$\Phi^{i j} D_{i j}u = f$ in $U$
with $|u|\leq 1$ in $U$ and $f\in L^q(U)$  $(q>n/2)$.   Assume $0<\alpha_1<1$ and $h\in W^{2,n}_{loc}(U_{\alpha_1})\cap C(\overline U_{\alpha_1})$  is a solution of
\begin{equation}\label{goodeq}
\left\{\begin{array}{rl}
\calW^{i j} D_{i j}h  &= 0 \quad \mbox{ in }\quad U_{\alpha_1}\\
h  &=u\quad\mbox{ on } \quad\partial U_{\alpha_1}.
\end{array}\right.
\end{equation}
Then, there exists $\gamma \in (0,1)$ depending only on $n$  and $q$ such that for any $0<\alpha_2<\alpha_1$ we have
\[
\|u-h\|_{ L^\infty(U_{\alpha_2})}
+\|f-\trace([\Phi-\calW ]D^2h)\|_{ L^q(U_{\alpha_2})}
\leq C(\alpha_1,\alpha_2,n,q) \left\{\| \Phi - \calW\|_{L^q(U_{\alpha_1})}^\gamma + \|f\|_{L^q(U)} \right\}
\]
provided that 
$\| \Phi - \calW\|_{L^q(U_{\alpha_1})} \leq (\alpha_1 -\alpha_2)^{\frac{2n}{1+ (n-1)\gamma}}$. 
\end{lemma}
Lemma \ref{explest} is an extension  of  \cite[Lemma~4.1]{GN2}. Its proof is omitted since it is similar to that of  \cite[Lemma~4.1]{GN2}. 
Instead of using the ABP estimate and interior H\"older estimate for 
equation \eqref{LMA-eq} with $L^n$ right hand side as in \cite{GN2}, we use Lemma~\ref{lm:inter-maximum-prin} and Corollary~\ref{cor:imterior-holder} for the linearized Monge-Amp\`ere
equation with $L^q$ right hand side.

We close this section by a result about the stability of cofactor matrices, which is a consequence of \cite[Lemma~3.5]{GN1} and \cite[Proposition~3.14]{LN2}.
\begin{lemma}\label{lm:convercofactors}
 Let $\Omega$ be a {\it normalized} convex domain. Let 
 $\phi, w\in C(\overline\Omega)$ be convex functions satisfying  
 $$1-\theta\leq \det D^2 \phi \leq 1+\theta~\text{in}~\Omega,~ \det D^2 w= 1~\text{in}~\Omega~\text{and }\phi=w=0 \text{ on } \partial\Omega.$$
Then for any $q\geq 1$, there exist $\theta_0>0$ and $C>0$ depending only on $q$ and $n$
such that
\begin{equation*}
\|\Phi - \calW\|_{L^q(B_\frac12(0))} \leq C \theta^{\frac{(n-1)\delta}{n(2nq-\delta)}}\quad  \mbox{for all } \theta\leq \theta_0,
\end{equation*}
where $\delta =\delta(n)>0$, and $\Phi$, $\calW$ are the matrices of cofactors of $D^2\phi$ and $D^2 w$, respectively.
\end{lemma}

\section{Pointwise $C^{1,\alpha}$ estimates in the interior and interior $W^{1,p}$ estimates}
\label{sec:interiorw1p}
 In this section, we sketch the proof of Theorem~\ref{HolderDv} and then use it to prove Theorem~\ref{Interior-Du}.
  
For the proof of Theorem \ref{HolderDv}, we need the next two lemmas from \cite{GN1} about  geometric properties of sections of solutions to the Monge-Amp\`ere 
equation. For a strictly convex function $\phi$ defined on $\Omega$ and $t>0$, we denote
 by $S_t(\phi)$ the section of $\phi$ centered at its minimum point with height $t$, i.e.,
 $$S_t(\phi):= \{x\in\Omega: \phi(x) \leq \min_{\Omega} \phi + t\}.$$
 We denote by $I$ the identity matrix.
 \begin{lemma} (\cite[Lemma 3.2]{GN1})
 \label{lem32}
  Suppose $B_1(0)\subset\Omega\subset B_n(0)$ is a {\it normalized} convex domain. Then there exist universal constants $\mu_0>0$, $\tau_0>0$ and a positive definite matrix $M= A^t A$ and $p\in \R^n$
  satisfying
  $$\det M=1,\quad  0<c_1I\leq M\leq c_2I,\quad\text{and}\quad |p|\leq c,$$
  such that if $\phi\in C(\overline{\Omega})$ is a strictly convex function in $\Omega$ with
  $$1-\e\leq \det D^2 \phi \leq 1+\e~\text{in}~\Omega,~\text{and }\phi =0 \text{ on } \partial\Omega, $$
  then for $0<\mu\leq\mu_0$ and $\e\leq \tau_0 \mu^2$, we have
  $$B_{\left(1-C(\mu^{1/2} + \mu^{-1}\e^{1/2})\right)\sqrt{2}} (0)\subset \mu^{-1/2}TS_\mu (\phi)\subset B_{\left(1+C(\mu^{1/2} + \mu^{-1}\e^{1/2})\right)\sqrt{2}} (0),$$
  and
  $$\left|\phi(x)- \Big[\phi(x_0) + p\cdot (x-x_0) +\frac{1}{2}\langle M(x-x_0), (x-x_0)\rangle\Big]\right|\leq C(\mu^{3/2}+\e)~\text{in}~S_\mu(\phi),$$
  where $x_0\in\Omega$ is the minimum point of $\phi$ and $Tx:= A(x-x_0)$.
 \end{lemma}
 \begin{lemma} (\cite[Lemma 3.3]{GN1})
 \label{lem33}
  Suppose $B_{(1-\sigma)\sqrt{2}}(0)\subset\Omega\subset B_{(1+\sigma)\sqrt{2}}(0)$ is a convex domain where $0<\sigma\leq 1/4$. Then there exist universal constants $\mu_0>0$, $\tau_0>0$
   which are independent of $\sigma$, a positive definite matrix $M= A^t A$, and $p\in \R^n$
  with
  $$\det M=1,\quad  (1-C\sigma)I\leq M\leq (1+ C\sigma)I,\quad \text{and} \quad |p-x_0|\leq C\sigma,$$
  such that if $\phi\in C(\overline{\Omega})$ is a strictly convex function in $\Omega$ with
  $$1-\e\leq \det D^2 \phi \leq 1+\e~\text{in}~\Omega,~\text{and }\phi=0 \text{ on } \partial\Omega, $$
  then for $0<\mu\leq\mu_0$ and $\e\leq \tau_0 \mu^2$, we have
  $$B_{\left(1-C(\sigma\mu^{1/2} + \mu^{-1}\e^{1/2})\right)\sqrt{2}} (0)\subset \mu^{-1/2}TS_\mu (\phi)\subset B_{\left(1+C(\sigma\mu^{1/2} + \mu^{-1}\e^{1/2})\right)\sqrt{2}} (0),$$
  and
  $$\left|\phi(x)- \Big[\phi(x_0) + p\cdot (x-x_0) +\frac{1}{2}\langle M(x-x_0), (x-x_0)\rangle\Big]\right|\leq C(\sigma\mu^{3/2}+\e)~\text{in}~S_\mu(\phi),$$
  where $x_0\in\Omega$ is the minimum point of $\phi$ and $Tx:= A(x-x_0)$.
 \end{lemma}
 We also use the following classical Pogorelov's estimates \cite[formula (4.2.6)]{G01}, and interior $C^{1,1}$ estimates for linear, uniformly elliptic equations \cite[Theorem
 6.2]{GiT};  see also the proof of \cite[Lemma~3.2]{GN1} and \cite[Theorem~2.7]{GN1}.
\begin{lemma}
 Suppose $B_1(0)\subset\Omega\subset B_n(0)$ is a {\it normalized} convex domain. Let $w$ be the convex solution to the equation $\det D^2 w=1$ in $\Omega$ with $w=0$
 on $\p\Omega$.
 \begin{myindentpar}{1cm}
  (i) Let $x_1\in\Omega$ be the minimum point of $w$. Then $|w(x_1)|\sim c_n$ for some universal constant $c_n$ and we have the Pogorelov's estimates
  $$\frac{2}{C_2^2} I\leq D^2 w(x)\leq \frac{2}{C_1^2} I~\text{ for all~}x\in \Omega~\text{with}~ dist(x,\p\Omega)\geq c_n,$$
where $C_1$ and $C_2$ are constants depending only on $n$.\\
 (ii) For any solution $h\in C^2 (B_1(0))$ of $\mathcal{L}_w h=0$ in $B_1(0)$, we have the classical interior $C^{1,1}$ estimate
 $$\|h\|_{C^{1,1}(B_{\frac12}(0))}\leq c_e \|h\|_{L^{\infty}(\partial B_{\frac34}(0))}$$
 for some constant $c_e$ depending only on $n$.
 \end{myindentpar}
\label{elliptic_est}
\end{lemma}

 \begin{proof}[Sketch of the proof of Theorem \ref{HolderDv}]
 Our proof utilizes results obtained in Section \ref{sec:MaximumHolder} together with the  arguments in the proof of \cite[Theorem~4.5]{GN1}. We sketch its proof here.
 Also for convenience,  we assume that the minimum point of $\phi$ is $\bar z =0$. \\
By diving our equation by $
K := \|u\|_{L^\infty(\Omega)} +\theta^{-1}N_{\phi, f, q}(0)$, we can  assume that
$$\Phi^{ij} u_{ij}(x)=f(x)\quad\mbox{in}\quad \Omega \quad\qquad\mbox{ with }\qquad \|u\|_{L^\infty(\Omega)}\leq 1,$$
and 
  $$\Big(\frac{1}{|S_r(\phi)|}\int_{S_r(\phi)}{|f|^q dx}\Big)^{\frac{1}{q}}\leq \theta r^{\frac{\alpha -1}{2}}\quad \mbox{for all}\quad 
S_r( \phi)\Subset\Omega \mbox{ with } r\leq r_0.$$

We need to prove that there exists an affine function $l(x)$ such that
\begin{equation}\label{equiconclu}
\sup_{0<r\leq \mu^*}{\Big( r^{-(1 + \alpha')} \|u-l\|_{L^\infty(B_r(0))}\Big)} + |l(0)| + \|D l(0)\|\leq C
\end{equation}
 with $\theta$, $\mu^*$ and $C$ depending only on $n$, $q$, $\alpha$, $\alpha'$ and  $r_0$. As in the proof of \cite[Theorem~4.5]{GN1}, (\ref{equiconclu}) follows from the following
 Claim.\\
{\bf Claim.}
There exist $0<\mu<1$ depending only on  $n$, $\alpha$ and $r_0$, a sequence of positive definite matrices $A_k$ with $\det A_k =1$ and a sequence of affine functions
$l_k(x) = a_k +  b_k \cdot x$ such that for all $k=1,2,3,\dots$ 
\begin{enumerate}
\item[(1)] $\|A_{k-1} A_k^{-1} \| \leq \frac{1}{\sqrt{c_1}}$, $\quad\|A_k\|\leq \sqrt{c_2 (1 +C \delta_0)(1 +C \delta_1)\cdots (1 +C \delta_{k-1})}$;
\item[(2)] $B_{(1-\delta_k)\sqrt{2}}(0)\subset  \mu^{\frac{-k}{2}}A_k S_{\mu^k}(\phi)\subset B_{(1+\delta_k)\sqrt{2}}(0)$;
\item[(3)] $\|u-l_{k-1}\|_{L^\infty(S_{\mu^k}(\phi))} \leq \mu^{\frac{k-1}{2} (1 + \alpha)}$;
\item[(4)] $|a_k - a_{k-1}| + \mu^{\frac{k}{2}}\|(A_k^{-1})^t \cdot (b_k - b_{k-1})\| \leq 2 c_e \mu^{\frac{k-1}{2} (1 + \alpha)}$,
\end{enumerate}
where 
\begin{align*}
& A_0 :=I,\quad l_0(x) := 0,\quad  \delta_0:=0;\quad \delta_1 := C(\mu^{1/2} + \mu^{-1} \theta^{1/2})<1- \frac{6}{5\sqrt{2}}, \quad\mbox{and}\\
&\delta_k := C(\delta_{k-1}\mu^{1/2} + \mu^{-1} \theta^{1/2})\quad\mbox{for } k\geq 2.
\end{align*}
Also  $C$,  $c_e$, $c_1$ and  $c_2$ are universal constants: $c_e$ is the constant in Lemma~\ref{elliptic_est};
 $c_1$ and $c_2$ are given by Lemma \ref{lem32} and $C$ is given by Lemma \ref{lem33}.  \\
The proof of the claim is by induction. It is quite similar to  the proof of 
\cite[Theorem~4.5]{GN1}. For reader's convenience, we indicate the proof for the cases $k=1,\, 2$. 

  Let $\mu_0>0$ and $\tau_0>0$ be the  small universal constants given by Lemma~\ref{lem32}. Let $0<\mu\leq \mu_0$ be fixed such that $\mu\leq r_0$, 
 $C_2 \sqrt{3\mu}\leq 1/2$  and $6 c_e C_2^2\mu^{\frac{1 -\alpha}{2}}\leq 1$,
where $C_2$ is the universal constant in the Pogorelov's estimates of Lemma \ref{elliptic_est}.
The constant $\theta\leq \theta_0$ will be determined later depending only on $n$, $q$, $\mu$, $\alpha$ and $\alpha'$, where $\theta_0=\theta_0(q,n)$ is given by Lemma~\ref{lm:convercofactors}. In particular by taking $\theta$ even smaller if necessary, we assume that
$\delta_1 = C(\mu^{1/2} + \mu^{-1} \theta^{1/2})<1- \frac{6}{5\sqrt{2}}$.

{\bf k=1:} Applying Lemma \ref{lem32} we obtain a positive definite matrix $M=A^t A$ with $\det A= \det M =1$, $c_1 I\leq M\leq c_2 I$ such that if we take $A_1:=A$ then 
\[
 B_{(1-\delta_1)\sqrt{2}}(0)\subset  \mu^{\frac{-1}{2}}A_1 S_{\mu}(\phi)\subset B_{(1+\delta_1)\sqrt{2}}(0),\quad\mbox{with}\quad
 \delta_1 := C(\mu^{1/2} + \mu^{-1} \theta^{1/2}).
\]

Then $(1)$ and $(2)$ hold obviously since  $\|A_1^{-1}\|\leq 1/\sqrt{c_1}$ and $\|A_1\|\leq \sqrt{c_2}$. Also $(3)$ is satisfied
as $l_0\equiv 0$ and $\|u\|_{L^\infty(\Omega)}\leq 1$.

{\bf k=2:} We first construct $l_1$ and verify $(3)$ for $k=2$ and $(4)$ for $k=1$. Then we construct $A_2$ and verify
$(1)$ and  $(2)$ for $k=2$.

+ Constructing $l_1(x)$: Recall that $D\phi(0)=0$ since the origin is the minimum point of $\phi$. 
Hence $S_\mu(\phi) =\{y\in \Omega:~ \phi(y)-\phi(0)-\mu\leq 0 \}$. Let  $\Omega_1^*:=\mu^{\frac{-1}{2}}A_1 S_{\mu}(\phi)$, and
\[
\phi^*(y) := \frac{1}{\mu}\big[\phi(\mu^{\frac{1}{2}}A_1^{-1} y) -\phi(0) -\mu \big], \quad 
v(y) := (u-l_0)(\mu^{\frac{1}{2}}A_1^{-1} y) =u(\mu^{\frac{1}{2}}A_1^{-1} y)
\]
for $y\in  \Omega_1^*$.
Then, as 
$
D^2\phi^*(y) =(A_1^{-1})^t D^2\phi(\mu^{\frac{1}{2}}A_1^{-1} y)A_1^{-1}$,
we get
\begin{equation*}
\left\{\begin{array}{rl}
1-\theta \leq\! \det D^2 \phi^* \!\!\!\!\!&\leq 1+\theta \quad\mbox{in}\quad \Omega_1^*\\
\phi^*\!\!\!\!\!&=0\ \ \ \ \ \quad\mbox{ on}\quad \partial \Omega_1^*
\end{array}\right.
\end{equation*}
Let
$\Phi^*(y) \equiv \left(\Phi^{*ij}(y)\right):= \det D^2\phi^*(y) ~ (D^2\phi^*(y))^{-1}.$ 
Then, by (\ref{trans1}), we get
\begin{align*}
 \Phi^{*ij} v_{ij}(y)
  = \mu~ f\big(\mu^{\frac{1}{2}}A_1^{-1} y\big) =: \tilde f(y) \qquad\qquad \mbox{in }\quad \Omega_1^*.
\end{align*}
Notice that, from $\det A_1=1$ and $\Omega_1^*:=\mu^{\frac{-1}{2}}A_1 S_{\mu}(\phi)$, we have
\begin{align*}
\left(\frac{1}{|\Omega_1^*|} \int_{\Omega_1^*}{|\tilde f(y)|^q dy}\right)^{\frac{1}{q}}
&=\mu~\left(\frac{1}{|S_\mu(\phi)|}\int_{S_\mu(\phi)}{|f(x)|^q dx}\right)^{\frac{1}{q}}
\leq \mu \theta \mu^{\frac{\alpha -1}{2}} = \theta \mu^{\frac{1+\alpha}{2}}.
\end{align*}
We apply Lemma~\ref{explest}  with $\phi\rightsquigarrow \phi^*$,
$f\rightsquigarrow \tilde f$,
$u\rightsquigarrow v$  and $U\rightsquigarrow \Omega_1^*$. 
Note that by $(3)$ we have $\|v\|_{L^\infty(\Omega_1^*)}\leq 1$. Recall that $\theta\leq \theta_0$, where $\theta_0$ is 
the small constant given by  Lemma~\ref{lm:convercofactors}. Hence if $h$ is the solution of 
\begin{equation*}
\left\{\begin{array}{rl}
\!\!\calW^{i j} D_{i j}h\!\!\!\!\!&= 0 \, \quad \mbox{ in}\quad S_\frac12(\phi^*)\\
\!\!h\!\!\!\!\!&=v\ \ \ \mbox{ on}\quad \partial S_\frac12(\phi^*)
\end{array}\right.
\quad\mbox{where}\quad 
\left\{\begin{array}{rl}
\!\!\det D^2 w \!\!\!\!\!&= 1\ \ \mbox{ in}\quad \Omega_1^*\\
\!\!w\!\!\!\!\!&=0\ \ \mbox{ on}\quad \partial\Omega_1^*,
\end{array}\right.
\end{equation*}
then
\begin{align*}
\|v-h\|_{L^\infty(S_\frac14(\phi^*))}
&\leq 
 C(n,q) \left\{\| \Phi^* - \calW\|_{L^q(S_\frac12(\phi^*))}^\gamma + \|\tilde f\|_{L^q(\Omega_1^*)} \right\}\\
 &\leq C(n,q) \left\{C \theta^{\frac{(n-1)\gamma \delta}{n(2nq-\delta)}} + \theta \mu^{\frac{1+\alpha}{2}} \right\}
\leq \frac12 \mu^{\frac{1+\alpha}{2}}.
\end{align*}
We have $\|h\|_{L^\infty(B_1)}\leq 1$ by the maximum principle. Moreover, it follows from the formulas
\cite[(3.13) and (3.15)]{GN1} that, for some $C_3=C_3(n)$, 
\[
S_{2\mu}(\phi^*) \subset B_{C_2 \sqrt{2\mu + C_3 \theta^{1/2}}}(0)\subset B_{C_2 \sqrt{3\mu}}(0).
\]
Thus, by letting  $\bar l(y) := h(0) + D h(0)\cdot y$, applying the interior $C^{1,1}$ estimate for $h$ as in Lemma \ref{elliptic_est} and noting that $C_2\sqrt{3\mu}\leq 1/2$, we get
\begin{align*}
\| h - \bar l\|_{L^\infty(S_{2\mu}(\phi^*))}
\leq \| h - \bar l\|_{L^\infty(B_{C_2 \sqrt{3\mu}}(0))}
 \leq  3 c_e C_2^2 \mu.
\end{align*}
 Therefore,
 \begin{align}\label{k=2est}
\| v - \bar l\|_{L^\infty(S_{2\mu}(\phi^*))} 
\leq \| v - h\|_{L^\infty(S_{2\mu}(\phi^*))} + \| h - \bar l\|_{L^\infty(S_{2\mu}(\phi^*))}
\leq \frac12 \mu^{\frac{1+\alpha}{2}}  + 3 c_e C_2^2 \mu\leq  \mu^{\frac{1}{2} (1+\alpha)}.
\end{align}
Define
\begin{equation}\label{defl_1}
l_1(x) := l_0(x) + \bar l(\mu^{\frac{-1}{2}}A_1  x).
\end{equation}
Then since $S_{\mu}(\phi^*) = \mu^{\frac{-1}{2}}A_1 S_{\mu^2}(\phi)$, we obtain from \eqref{k=2est} for $x\in S_{\mu^2}(\phi)$ that
 \begin{align*}
|u(x)-l_1(x)| 
= |v(\mu^{\frac{-1}{2}}A_1 x) -\bar l(\mu^{\frac{-1}{2}}A_1  x)|
\leq \| v - \bar l\|_{L^\infty(S_{\mu}(\phi^*))} 
\leq  \mu^{\frac{1}{2} (1+\alpha)}.
\end{align*}
Thus $(3)$ for $k=2$ is verified. Also $(4)$ for $k=1$ holds because it follows from the definition \eqref{defl_1} and the definition of $\bar l$ that
$a_1 = a_0 + h(0)$ and $b_1= b_0 + \mu^{\frac{-1}{2}}A_1^t D h(0)$. Hence by using the interior $C^{1,1}$ estimate for $h$, we get  $(4)$ for $k=1$ from
 \begin{align*}
|a_1 - a_0| + \mu^{\frac{1}{2}} \|(A_1^{-1})^t \cdot ( b_1 - b_0)\| = |h(0)| + \|D h(0)\|
\leq 2 c_e.
\end{align*}

+ Constructing $A_2$: Applying Lemma \ref{lem33} for $\phi^*$  and $\Omega_1^*$  we obtain a positive definite matrix $M=A^t A$ with $\det M =1$, $(1- C\delta_1) I\leq M\leq (1+C\delta_1)I$ such that 
\[
 B_{(1-\delta_2)\sqrt{2}}(0)\subset  \mu^{\frac{-1}{2}}A S_{\mu}(\phi^*)\subset B_{(1+\delta_2)\sqrt{2}}(0),\quad\mbox{with}\quad \delta_2 := C(\delta_1\mu^{1/2} + \mu^{-1} \theta^{1/2}).
\]
Define $A_2:=A A_1$ which implies in particular that $A_2$ is a positive definite matrix with $\det A_2 =1$. Then  as $S_{\mu}(\phi^*) = \mu^{\frac{-1}{2}}A_1 S_{\mu^2}(\phi)$ we conclude that
\[
 B_{(1-\delta_2)\sqrt{2}}(0)\subset  \mu^{-1}A_2 S_{\mu^2}(\phi)\subset B_{(1+\delta_2)\sqrt{2}}(0).
\]
Thus $(2)$ and the first part of $(1)$ for $k=2$ hold obviously since $A_1 A_2^{-1} =A^{-1}$ and $\|A^{-1}\|\leq 
\frac{1}{\sqrt{1-C\delta_1}}\leq \frac{1}{\sqrt{c_1}}$. Next observe from the definition of $A$ that $(1-C\delta_1) |x|^2 \leq |A x|^2\leq (1+C\delta_1)|x|^2$. Hence 
\[
|A_2 x|^2= |A A_1 x|^2 \leq (1+C\delta_1)|A_1 x|^2 \leq c_2 (1+C\delta_1)|x|^2
\]
yielding  the second part of $(1)$, i.e., $\|A_2\|\leq \sqrt{c_2 (1+C\delta_1)}$. 
\end{proof}

We next prove Theorem~\ref{Interior-Du}, and in this proof we use the following strong type inequality for the maximal function with respect to sections:
\begin{theorem} (\cite[Theorem 2.2]{GN2})\label{strongtype1} 
Assume that $\Omega$ and $\phi$ satisfy {\bf (H)}.
Let $\Omega'\Subset \Omega$. Fix $h_0>0$ such that
$S_\phi(x, 2h_0)\Subset\Omega$ for all $x\in\Omega'$. Define the maximal function $\mathcal{M}(f)$ by
$$
\mathcal M (f)(x)=\sup_{t\leq h_0}
\dfrac{1}{|S_\phi(x,t)|}\int_{S_\phi(x,t)}|f(y)|\, dy\quad \text{for } x\in \Omega'.
$$
For any $1<p<\infty$, there exists a constant $C$ depending on $p$, $n$, $\lambda$, $\Lambda$ and 
$\dist(\Omega', \p\Omega)$ 
such that
\[
\left(\int_{\Omega'}{|\mathcal M
(f)(x)|^p ~d\mu(x)} \right)^{\frac{1}{p}}\leq
C\, \left(\int_{\Omega}{|f(y)|^p \,d\mu(y)} \right)^{\frac{1}{p}}.
\]
\end{theorem}

\smallskip 
\begin{proof}[Proof of Theorem~\ref{Interior-Du}]
Let  $\Omega'\Subset \Omega$. Let $0<\alpha<1$, $q'$ be such that $n/2 <q'<q$, and
\[
N(z) :=\sup_{r\leq h_0}  r^{\frac{1-\alpha}{2}}\Big(  \frac{1}{|S_\phi(z,r)|}\int_{S_\phi(z,r)} |f|^{q'}\, dx\Big)^{\frac{1}{q'}}
\]
where $h_0$ is to be determined. One of the requirements is that $S_\phi(y, 2 h_0)\Subset\Omega$ for all $y\in\Omega'$.
Then we have the following pointwise estimate for the gradient $Du$:
\begin{equation}\label{iteration-consequence}
|Du(y)|
\leq C \Big[ \|u\|_{L^\infty(\Omega)} + N(y)\Big]\quad \mbox{for a.e. } y\in\Omega'.
\end{equation}
The $L^p$ estimate \eqref{interiorW1p-est} for $Du$ 
then follows from the volume growth of interior sections of $\phi$ and the strong type inequality for the maximal function $\mathcal{M}(f)$ in 
Theorem \ref{strongtype1}.
Indeed, by H\"older inequality, it suffices to consider the case $q\leq p<\frac{nq}{n-q}$. 
From  \eqref{iteration-consequence} and by using H\"older inequality, we have for any $p\geq q$ that
\begin{align*}
&\|Du\|_{L^p(\Omega')}\leq  C  \|u\|_{L^\infty(\Omega)} + C \|N\|_{L^p(\Omega')} \\
&\leq C  \|u\|_{L^\infty(\Omega)}
+C\left( \int_{\Omega'}\sup_{r\leq h_0} \Big\{ r^{\frac{p}{2}(1-\alpha)}\M(f^{q'})(y)^{\frac{q}{q'}} \Big(\frac{1}{|S_\phi(y,r)|} \int_{S_\phi(y,r)}{|f(x)|^{q'} dx} \Big)^{\frac{p-q}{q'}}\Big\} dy\right)^{\frac{1}{p}}\\
&\leq C  \|u\|_{L^\infty(\Omega)}
+C\left( \int_{\Omega'}\sup_{r\leq h_0} \Big\{ r^{\frac{p}{2}(1-\alpha)}\M(f^{q'})(y)^{\frac{q}{q'}} \Big(\frac{1}{|S_\phi(y,r)|} 
\int_{S_\phi(y,r)}{|f(x)|^q dx} \Big)^{\frac{p-q}{q}}\Big\} dy\right)^{\frac{1}{p}}
\\ &\leq C\|u\|_{L^\infty(\Omega)}
+C \sup_{r\leq h_0} \Big\{ r^{\frac12\big[(1-\alpha) -\frac{ n}{q} +\frac{n}{p}\big]}  \Big\}
\left( \int_{\Omega'} \M(f^{q'})(y)^{\frac{q}{q'}}  dy\right)^{\frac{1}{p}}
\|f\|_{L^{q}(\Omega)}^{\frac{p-q}{p}} .
\end{align*}
The last inequality above follows from the volume estimates of interior sections of $\phi$. These estimates
\cite[Corollary 3.2.4]{G01} say that there exist constants $C$ and $C'$ depending only on $n,\lambda,\Lambda$ such that $$Cr^{n/2}\leq |S_\phi(y, r)|\leq C' r^{n/2}~ \text{for all}~
y\in\Omega' ~\text{and }~ r\leq h_0.$$ 
As $q/q'>1$, we can apply Theorem \ref{strongtype1} to conclude that
\begin{eqnarray*}
\|Du\|_{L^p(\Omega')} &\leq&  C  \|u\|_{L^\infty(\Omega)}
+C\sup_{r\leq h_0} \Big\{ r^{\frac12\big[(1-\alpha) -\frac{ n}{q} +\frac{n}{p}\big]}  \Big\}\|f\|_{L^q(\Omega)}^{\frac{q}{p}}
\|f\|_{L^{q}(\Omega)}^{\frac{p-q}{p}}\\
&=& C  \|u\|_{L^\infty(\Omega)} + C\sup_{r\leq h_0} \Big\{ r^{\frac12\big[(1-\alpha) -\frac{ n}{q} +\frac{n}{p}\big]}  \Big\} 
\|f\|_{L^{q}(\Omega)}.
\end{eqnarray*}
Now, since $q\leq p< \frac{nq}{n-q}$, we can choose $\alpha\in (0, 1-\frac{n}{q}+\frac{n}{p})$ to obtain estimate \eqref{interiorW1p-est}:
$$\|Du\|_{L^p(\Omega')}  \leq C \Big(  \|u\|_{L^\infty(\Omega)}
+ \|f\|_{L^q(\Omega)} \Big).$$
It remains to prove (\ref{iteration-consequence}). Given $\eps_0>0$, since $g\in C(\Omega$) and by \cite[Theorem 3.3.8]{G01}, there exists $h_0>0$ such that for any $y\in \Omega'$,
\[
B_{C_1 h_0}(y)\subset S_\phi (y, h_0) \subset B_{C_2 h_0^b}(y) \quad\mbox{and}\quad  |g(x)- g(y)|\leq \e_0 \quad \text{for all } x\in S_\phi(y, h_0).
\]
Fix $y\in \Omega'$, and let $Tx = A(x-y) + \bar z$ be an affine transformation
 such that 
 \[B_1(0)\subset T S_\phi(y, h_0)  \subset B_n (0).
 \]
Notice that $C^{-1}\leq |\det A|^{\frac{2}{n}} h_0 \leq C$ for some constant $C>0$ depending only on $n$, $\lambda$ and $\Lambda$.

 Define $\tilde \Omega := T S_\phi(y, h_0)$  and consider the functions
\begin{align*}
\tilde \phi(z) := \kappa \big[ \phi(T^{-1} z) - l_{y} (T^{-1} z) - h_0 \big] \quad\mbox{and}\quad
\tilde u(z) := g(y)\kappa^{\frac{\alpha - 3}{2}} u(T^{-1} z), \quad\mbox{for}\quad z\in \tilde \Omega
\end{align*}
where  $\kappa := g(y)^{\frac{-1}{n}} |\det A|^{\frac{2}{n}}$ and $l_y(x)$ is the supporting function of $\phi$ at $y$.
Then
\begin{align*}
1-\frac{\e_0}{\lambda} \leq \det D^2 \tilde \phi(z)\leq 1+ \frac{\e_0}{\lambda} 
\quad\mbox{and}\quad
\tilde\Phi^{i j}  \tilde u_{i j}(z) = \kappa^{\frac{\alpha - 1}{2}} f(T^{-1} z) =:\tilde f(z)\quad \mbox{in}\quad \tilde \Omega.
\end{align*}
We have
\begin{align*}
 r^{\frac{1-\alpha }{2}} \left(\frac{1}{|S_{\tilde\phi}(\bar z,r )|} \int_{S_{\tilde \phi}(\bar z, r)}{|\tilde f|^{q'} dz}\right)^{\frac{1}{q'}}
=(\kappa^{-1} r)^{\frac{1-\alpha }{2}} \left(\frac{1}{|S_\phi (y, \kappa^{-1} r)|}\int_{S_\phi (y, \kappa^{-1} r)}{|f|^{q'} dx}\right)^{\frac{1}{q'}}
\end{align*}
for all $r\leq \kappa h_0$.
Since  $\kappa h_0 =g(y)^{\frac{-1}{n}} |\det A|^{\frac{2}{n}} h_0\geq c(n,\lambda, \Lambda) >0$, it follows by letting 
$r_0:=c(n,\lambda, \Lambda)$ that $N_{\tilde \phi, \tilde f, q'}(\bar z)\leq N(y)$.
Note that $\bar z$ is the minimum point of $\tilde \phi$ in $\tilde\Omega$ .
 Therefore if we choose $\e_0 := \lambda \theta$, where $\theta>0$ is the constant given in Theorem~\ref{HolderDv} corresponding to this  $r_0$, $\alpha'=0$  and $q\rightsquigarrow q'$, then by Theorem~\ref{HolderDv} there exist constants $\mu^*, C >0$ depending only on $n$, $q'$, $\alpha$, $\lambda$ and $\Lambda$, and an affine function $\bar l$ such that
  \begin{equation}\label{eq:yHolder}
  |\tilde u(z) - \bar l(z)| + |z-\bar z| |D \bar l | \leq  C  |z-\bar z| \Big[ \|\tilde u \|_{L^\infty(\tilde \Omega)} +N(y) \Big]  \quad \mbox{for all}\quad z \in B_{\mu^*}(\bar z) \Subset \tilde\Omega.  
  \end{equation}
  Observe that as $B_{C_1 h_0}(y)\subset S_\phi (y, h_0)$, we have $T B_{C_1 h_0}(y)\subset B_n(0)$, i.e.,  $A B_{C_1 h_0}(0) +\bar z\subset B_n(0)$. This  yields  
  $\| A\|\leq C h_0^{-1}$.
Thus $T B_{C^{-1}\mu^* h_0}(y) \subset B_{\mu^*}(\bar z)$ and
we obtain from \eqref{eq:yHolder} and by rescaling back  and by taking  $\ell(x) =\ell(x,y) := g(y)^{-1} \kappa^{\frac{3-\alpha}{2}} \bar l (T x)$ that
 \begin{align*}
  &|u(x) - \ell(x)| + |x-y| |D \ell |
  = g(y)^{-1} \kappa^{\frac{3-\alpha}{2}} \Big[  |\tilde u(T x)- \bar l(T x)|  +|x-y| | D \bar l \cdot A|\Big]\\
&\leq C \|A\|
|x- y| g(y)^{-1} \kappa^{\frac{3-\alpha}{2}}  \Big[ \|\tilde u \|_{L^\infty(\tilde \Omega)} + N(y) \Big]\\
&= C \|A\|
|x- y| \Big[ \|u\|_{L^\infty(S_\phi(y, h_0))} +  g(y)^{-1} \big( g(y)^{\frac{-1}{n}} |\det A|^{\frac{2}{n}} \big)^{\frac{3-\alpha }{2}} N(y)\Big]\\
 &\leq C h_0^{-1} h_0^{\frac{\alpha - 3}{2}}   |x- y|
\Big[ \|u\|_{L^\infty(S_\phi(y, h_0))} + N(y)\Big]
\qquad\mbox{for all}\quad x  \in B_{C^{-1}\mu^* h_0}(y).
 \end{align*}
In other words, we proved that for any $y\in \Omega'$ there exists
an affine  function $\ell$ such that 
\begin{equation}\label{eq:holderestimatefirstderivatives}
|u(x)-\ell(x)| + |x-y| | D\ell |\leq C  h_0^{\frac{\alpha  - 5}{2}} |x-y|\Big[ \|u\|_{L^\infty(\Omega)} + N(y)\Big]
\qquad\text{for all }  x  \in B_{C^{-1}\mu^* h_0}(y).
\end{equation}
Now, let $y\in\Omega'$ be such that $Du(y)$ exists. Then  by using \eqref{eq:holderestimatefirstderivatives} we get
\begin{align*}
 |u(x)-u(y)|
 &\leq |u(x) -\ell(x)| + |\ell(x)- \ell(y)| + |u(y)-\ell(y)|\\
  &\leq C |x-y| \Big[ \|u\|_{L^\infty(\Omega)} + N(y)\Big] \qquad\mbox{for all }\quad x  \in B_{C^{-1}\mu^* h_0}(y),
\end{align*}
which yields \eqref{iteration-consequence}. Note that the constant $C$ depends also on $h_0$, and hence it depends on the modulus of continuity of $g$. 
\end{proof}

\section{Pointwise $C^{1,\alpha}$ estimates at the boundary}
\label{sec:boundaryw1p}
In this section, we prove Lemma \ref{lm:glob-maximum-prin2} and Theorem~\ref{h-bdr-gradient}.
The proof of Theorem~\ref{h-bdr-gradient} 
is similar to that of \cite[Theorem~1.1]{LS2} 
but we include it here for the sake of completeness. It
uses the perturbation arguments in the spirit of Caffarelli \cite{C89, CC} (see also Wang \cite{Wang})
and boundary H\"older gradient estimates for the case 
of bounded right hand side 
$f$ and $C^{1,1}$ boundary data by Savin and the first author \cite{LS1}. We recall these estimates in the following theorem.
\begin{theorem}(\cite[Theorem 2.1 and Proposition 6.1]{LS1})
Assume $\phi$ and $\Omega$ satisfy assumptions 
\eqref{om_ass}--\eqref{eq_u1}. Denote for simplicity $S_t = S_\phi(0, t)$. Let $u: S_r\cap 
\overline{\Omega}\rightarrow \R$ be a continuous solution to 
\begin{equation*}
\Phi^{ij}u_{ij}  = f ~\text{in} ~ S_r\cap \Omega, ~\text{and}~
u = 0~\text{on}~\p \Omega \cap S_r\end{equation*} 
where $f\in L^{\infty}(S_r\cap\Omega)$.
Then, for all $s\leq r/2$, we have  
$$|\p_n u(0)| + s^{-\frac{1+\alpha_0}{2}}\max_{S_s}|u-\p_n u(0)x_n| \leq C_0 \left(\|u\|_{L^{\infty}(S_r\cap\Omega)} + \|f\|_{L^{\infty}(S_r\cap\Omega)}\right)$$
where $\alpha_0\in (0, 1)$ and $C_0$ are constants depending only on $n, \rho,
 \lambda, \Lambda $.

\label{LS-gradient}
\end{theorem}

Assume $\phi$ and $\Omega$ satisfy
\eqref{om_ass}--\eqref{eq_u1}. We can also assume that $\phi(0)=0$ and $D \phi(0)=0.$ 

By Savin's Localization Theorem for solutions to the Monge-Amp\`ere equations proved in \cite{S1, S2}, there exists a small constant $k$ 
depending only on $n, \rho, \lambda, \Lambda$ such that if $h\leq k$ then 
\begin{equation}kE_h\cap \overline{\Omega}\subset S_{\phi}(0, h)\subset k^{-1} E_h\cap \overline{\Omega}.
 \label{loc-k}
\end{equation}
Here
$E_h:= h^{1/2}A_h^{-1} B_1(0)$
with $A_h$ being a linear transformation (sliding along the $x_n=0$ plane) 
\begin{equation}A_h(x) = x- \tau_h x_n,~ \tau_h\cdot e_n =0, ~\det A_h =1\
 \label{Amap}
\end{equation}
and
$$|\tau_h|\leq k^{-1}\abs{log h}.$$
Let us write $\tau_h =(\nu_h,0)$ with $\nu_h\in \R^{n-1}$.
Next, we define the following rescaling of $\phi$
\begin{equation}\phi_h(x):= \frac{\phi(h^{1/2} A^{-1}_h x)}{h}
\quad \mbox{ in }\quad  \Omega_h:= h^{-1/2}A_h\Omega.
 \label{phiomega-h}
\end{equation}
Then
$$\lambda \leq \det D^2 \phi_h(x)= \det D^2 \phi(h^{1/2}A_{h}^{-1}x)\leq \Lambda\quad \mbox{in}\quad \Omega_h$$
and
$$B_{k}(0)\cap \overline{\Omega_h}\subset S_{\phi_h}(0, 1)= h^{-1/2} A_{h}S_\phi(0,h)\subset B_{k^{-1}}(0)\cap \overline{\Omega_h}.$$
We note that Lemma~4.2 in \cite{LS1} implies that if $h, r\leq c$ small then $\phi_h$ satisfies in $S_{\phi_h}(0, 1)$ the hypotheses of the Localization Theorem
\cite{S1, S2}
at all $x_0\in S_{\phi_h}(0, r)\cap\p S_{\phi_h}(0, 1).$ In particular, there exists  $\tilde\rho>0$ small  depending only on $n, \rho, \lambda, \Lambda$ such that
if $x_0\in S_{\phi_h}(0, r)\cap\p S_{\phi_h}(0, 1)$ then
\begin{equation}
 \tilde\rho\abs{x-x_{0}}^2 \leq \phi_h(x)- \phi_h(x_{0})-D \phi_h(x_{0}) \cdot (x- x_{0}) \leq 
\tilde\rho^{-1}\abs{x-x_{0}}^2,\quad \forall x \in \p S_{\phi_h}(0, 1).
\label{loc-h}
\end{equation}
Moreover, for $h, t\leq c$, we have the following volumes estimates
\begin{equation}
 \label{vol_ht}
 c_1 h^{\frac{n}{2}}\leq |S_\phi(0, h)|\leq C_1 h^{\frac{n}{2}}; \quad c_1 t^{\frac{n}{2}}\leq |S_{\phi_h}(0, t)|\leq C_1 t^{\frac{n}{2}}.
\end{equation}

 We fix $r$ in what follows. Then, the boundary H\"older gradient estimates in Theorem \ref{LS-gradient} for solutions to the linearized Monge-Amp\`ere equation  with 
bounded right hand side and $C^{1, 1}$ boundary data hold in $S_{\phi_h}(0, r)$.

We now employ the  Green's function estimate  obtained in \cite{Le15} to derive a boundary version of the generalized maximum principle in Lemma~\ref{lm:inter-maximum-prin}.

\begin{lemma}[Boundary maximum principle]\label{lm:glob-maximum-prin} 
Let  $h, t\leq c$ where $c= c(n,\lambda,\Lambda, \rho)$ is universally small. 
Let $f\in L^q(S_{\phi_h}(0, t))$ for some $q>n/2$  and  $u\in W^{2,n}_{loc}(S_{\phi_h}(0, t))\cap C(\overline{S_{\phi_h}(0, t)})$ satisfy 
\[
\mathcal{L}_{\phi_h} u\leq f\quad \mbox{almost everywhere in}\quad S_{\phi_h}(0, t).
\]
Then  there exists a constant $C>0$ depending only on  $n$, $\lambda$, $\Lambda$, $\rho$ and  $q$  such that
\[
\sup_{S_{\phi_h}(0, t)}{u} \leq \sup_{\partial S_{\phi_h}(0, t)}{u^+} + C |S_{\phi_h}(0, t)|^{\frac{2}{n} -\frac{1}{q}}   \|f\|_{L^q(S_{\phi_h}(0, t))}. 
\]
\end{lemma}
\begin{proof}
 Let $V=S_{\phi_h}(0, t)$. Let $G_V(\cdot, y)$ be the Green's function of $\mathcal{L}_{\phi_h}$ in $V$ with pole $y\in V$. As in \eqref{comparison-principle}, we obtain for
 all $x\in S_{\phi_h}(0, t)$ the estimate
 $$u(x) \leq \sup_{\p S_{\phi_h}(0, t)} u^{+} +\int_V G_V(x, y) f(y) dy.$$
 The conclusion of the lemma follows once we establish that for $q'=\frac{q}{q-1}$, we have
 \begin{equation}
 \label{glGV}
  \|G_V(x, \cdot)\|_{L^{q'}(V)}\leq C|V|^{\frac{2}{n}-\frac{1}{q}}\quad \text{for all~}x\in V.
 \end{equation}
Thanks to  \eqref{loc-h}, one can find a  constant $\theta_{\ast}>1$ depending only on $n,\lambda,\Lambda$ and $\rho$ such that
\begin{equation}\label{engulf1}
S_{\phi_h}(0, t)\subset S_{\phi_h}(x, \theta_\ast t)
 ~\text{for all}~x\in S_{\phi_h}(0, t).
\end{equation}
 This is a boundary version of the engulfing property of sections of the Monge-Amp\`ere equation (see
 \cite[Lemma~4.1]{LN1}). By the symmetry of the Green's function, we have
 \begin{eqnarray}
 \label{Gest1}
  \int_V G^{q'}_V(x, y) dy = \int_{V} G_V^{q'}(y, x) dy \leq \int_{S_{\phi_h}(x, \theta_\ast t)} G^{q'}_{S_{\phi_h}(x, \theta_\ast t)} (y, x) dy.
 \end{eqnarray}
Due to  $q'<\frac{n}{n-2}$, we have from  \cite[Corollary~2.6]{Le15} that
\begin{equation}
\label{Gphi}
 \int_{S_{\phi}(x, \theta_\ast t)} G^{q'}_{S_{\phi}(x, \theta_\ast t)} (y, x) dy \leq
  C(n, \lambda,\Lambda, \rho, q) |S_{\phi}(x, \theta_\ast t)|^{1-\frac{n-2}{n}q'}.
\end{equation}
By inspecting the proof of \cite[Corollary~2.6]{Le15} (see the discussion below), we see that the above inequality also holds with $\phi_h$ replacing $\phi$:
\begin{equation}
 \label{Gphih}
 \int_{S_{\phi_h}(x, \theta_\ast t)} G^{q'}_{S_{\phi_h}(x, \theta_\ast t)} (y, x) dy \leq
  C(n, \lambda,\Lambda, \rho, q) |S_{\phi_h}(x, \theta_\ast t)|^{1-\frac{n-2}{n}q'}.
  \end{equation}
The desired estimate \eqref{glGV} then follows from \eqref{Gest1}, \eqref{Gphih}, and the volume estimate for sections of $\phi_h$ given in \eqref{vol_ht}.  

Let us describe the proof of \eqref{Gphih}.
The difference  between \eqref{Gphih} and \eqref{Gphi} is that we only know $\phi_h$ and $S_{\phi_h}(0, 1)$ satisfying the quadratic separation condition 
(\ref{loc-h}) on a portion $S_{\phi_h}(0, r)\cap \p S_{\phi_h} (0, 1)$
of the boundary $\p S_{\phi_h} (0, 1)$ while $\phi$ and $\Omega$ satisfy a global condition. For reader's convenience, we indicate how to obtain (\ref{Gphih}) in our local setting from the proof of (\ref{Gphi}) in 
\cite[Corollary~2.6]{Le15}. Three main ingredients need to be verified are:
\begin{myindentpar}{1cm}
 (1) The engulfing property of sections: there exists some constant $\bar\theta= \bar\theta (n,\lambda,\Lambda, \rho)>1$
 such that if $x\in S_{\phi_h}(0, \delta)$ with $\delta$ universally small and $y\in S_{\phi_h}(x, t)$ with $t\leq c$,
 then we have
 \begin{equation}
  \label{engulf_h}
  S_{\phi_h}(x, t)\subset S_{\phi_h}(y, \bar\theta t).
 \end{equation}
(2) The volume growth of sections: if $x\in S_{\phi_h}(0, c)$ and $t\leq c$ then
\begin{equation*}
 C_1^{-1} t^{\frac{n}{2}}\leq |S_{\phi_h}(x, t)|\leq C_1 t^{\frac{n}{2}}.
\end{equation*}
(3) Boundary Harnack inequality for solutions to the homogeneous linearized Monge-Amp\`ere equation $\mathcal{L}_{\phi_h} v=0$ in $S_{\phi_h}(0, 1)$.
\end{myindentpar}
We now address these ingredients.

Concerning (1): Suppose $x\in S_{\phi_h}(0, \delta)$ and $y\in S_{\phi_h}(x, t)$. By (\ref{engulf1}), it suffices to consider $x\in
S_{\phi_h}(0,\delta)\cap\Omega_h$. We use the strict convexity result for $\phi_h$ (see \cite[Lemma~5.4]{LS1} and also \cite[Lemma~3.8(iv)]{Le15}) which says that
the maximal interior section $S_{\phi_h}(x, \bar h(x))$ of $\phi_h$ centered at $x$
where
$$\bar h(x)= \sup\{t|\,\,  S_{\phi_h}(x, t)\subset \Omega_h\}$$
is tangent to $\p\Omega_h$ at $z\in\p\Omega_h\cap S_{\phi_h}(0, r/2)$. Using equation (4.11) in the proof of Proposition 2.3 in \cite{LN1}, we find some $K=K(n,\lambda,\Lambda, \rho)$
such that
\begin{equation}
\label{dou_t}
 S_{\phi_h}(x, 2t)\subset S_{\phi_h} (z, Kt)~ \text{for all} ~\bar h(x) /2<t\leq c.
\end{equation}
If $t\leq \bar h(x)/2$, then $S_{\phi_h}(x, 2t)\subset \Omega_h$ and hence the inclusion (\ref{engulf_h}) follows from the 
engulfing property of interior sections
for the Monge-Amp\`ere equation with bounded right hand side (see the proof of Theorem~3.3.7 in \cite{G01}). Consider now $\bar h(x)/2<t\leq c$. Then we have from (\ref{dou_t})
$y\in S_{\phi_h}(z, Kt)$. By (\ref{engulf1}), we have $S_{\phi_h}(z, Kt)\subset S_{\phi_h}(y, \theta_{\ast} Kt)$. Recalling (\ref{dou_t}), we find that (\ref{engulf_h})
follows with $\bar\theta= \theta_{\ast}K$.

Concerning (2): The proof uses the Localization Theorem and (\ref{dou_t}) as in the proof of \cite[Corollary~2.4]{LN1} so we omit it.

Concerning (3): Given (1) and (2), the proof of the boundary Harnack inequality \cite[Theorem~1.1]{Le15} applies in our local setting without change.
\end{proof}

\begin{proof}[Proof of Lemma~\ref{lm:glob-maximum-prin2}]
The proof of this lemma is similar to that of Lemma~\ref{lm:glob-maximum-prin}. It uses the symmetry of the Green's function $G_\Omega(x,y)$ and 
its global integrability established in \cite[Corollary 2.6]{Le15} which says that  
for $p\in (1, \frac{n}{n-2})$ in the case $n\geq 3$ and $p\in (1,\infty)$ in the case $n=2$, we have 
$$
\sup_{x\in\Omega}\int_{\Omega}{G_\Omega(x,y)^p~dy}
\leq C(n,\lambda,\Lambda, p,\rho).
$$
\end{proof}

\smallskip
\begin{proof}[Proof of Theorem \ref{h-bdr-gradient}] Let $M:=\|\varphi\|_{C^{1,\gamma}(B_{\rho}(0)\cap\p\Omega)}$.
Since $u =\varphi$ on $\p\Omega \cap B_{\rho}(0)$, by subtracting a suitable affine function $l(x)$, we can assume that  $u$
satisfies
$\abs{u(x)}\leq M |x'|^{1+\gamma}$ for $x=(x', x_n)\in \p\Omega\cap B_{\rho}(0)$. In particular, $u(0)=0$.

Fix $0<\alpha<\min\{\gamma, \alpha_0\}$ where $\alpha_0$ is in Theorem~\ref{LS-gradient}. Let $\bar h\leq \theta^2$ with  
$\theta$ being some universally small constant that will be   chosen later.
Then by dividing our equation by 
\[K:= \theta^{-\frac{1+\alpha}{2}}\big[\|u\|_{L^{\infty}(B_{\rho}(0)\cap\Omega)}  +
\sup_{ \bar h\leq t\leq \theta^2} N_{\phi, f, q, 2\theta^{-1} t}(0) + M\big],
\]
 we may assume that 
\begin{equation}
\label{furthersimplify}
\|u\|_{L^{\infty}(B_{\rho}(0)\cap\Omega)}  + \sup_{\bar h\leq t\leq \theta^2}N_{\phi, f, q, 2\theta^{-1} t}(0) + M\leq (\theta^{1/2})^{1+\alpha}=: \delta,
\end{equation}
and we only need to show that there exists $b\in \R^n$ such that
\begin{equation}
\label{equivalent-est}
\bar h^{-\frac{1+\alpha}{2}}\|u -bx\|_{L^{\infty}(S_{\phi}(0, \bar h))} + \|b\|
 \leq C(n,q,\rho,\alpha,\gamma,\lambda, \Lambda). 
\end{equation}
As a consequence of \eqref{furthersimplify}, we have
\begin{equation}\label{boundary-u}
\abs{u(x)}\leq \delta |x'|^{1+\gamma}\quad \mbox{for}\quad x=(x', x_n)\in \p\Omega\cap B_{\rho}(0).
\end{equation}
{\bf Claim.} There exist $\theta>0$ small and $C_2>1$ depending only on $n, \rho, \lambda, \Lambda, \gamma, q$ such that the following holds. 
If $\sup_{1\leq m\leq k}N_{\phi, f, q, 2\theta^m}(0)\leq C_2 \delta$ for some integer number $k\geq 2$, then for every $m=1, 2, \dots,k$  we can find 
a linear function
$l_m(x):= b_m x_n$
with  $b_0= b_1 =0$ such that 
\begin{enumerate}
 \item[(i)] 
 $\|u-l_m\|_{L^{\infty}(S_{\theta^m})}\leq (\theta^{m/2})^{1+\alpha};$
\item[ (ii)] $|b_m-b_{m-1}|\leq C_0 (\theta^{\frac{m-1}{2}})^{\alpha}.$
\end{enumerate}
The desired estimate \eqref{equivalent-est} follows from the above claim. Indeed, since $\bar h\leq \theta^2$ 
we can find a positive integer $k\geq 2$ such that $\theta^{k+1}<\bar h\leq \theta^k$
and the conclusion \eqref{equivalent-est} follows by choosing $b= b_k$. To see this, we use the definition of $N$ in (\ref{Nr1def}) and 
$2\theta^k< 2\theta^{-1}\bar h\leq 2\theta^{k-1}$, together with the volume estimate \eqref{vol_ht}  to get
$$N_{\phi, f, q, 2\theta^k}(0) \leq C_2 N_{\phi, f, q, 2\theta^{-1}\bar h}(0)$$
for some universal constant $C_2$. This and  \eqref{furthersimplify} imply that
\[
\sup_{1\leq m\leq k}N_{\phi, f, q, 2\theta^m}(0) \leq C_2 \sup_{\bar h\leq t\leq \theta^2}N_{\phi, f, q, 2\theta^{-1}t}(0)\leq C_2 \delta.
\]
Hence we deduce from the claim by  taking into account the affine function $l_k$ that 
\begin{align*}(\theta^k)^{-\frac{1+\alpha}{2}}\|u- b_k x\|_{L^{\infty}(S_{\theta^k})} + \|b_k\|
\leq 1+  \sum_{m=1}^k |b_m - b_{m-1}| \leq  1+  C_0 \sum_{m=1}^\infty \theta^{\frac{\alpha}{2}(m-1)}\leq C.
 \end{align*}
 Therefore, we obtain \eqref{equivalent-est} with $b=b_k$ since
\begin{align*}\bar h^{-\frac{1+\alpha}{2}}\|u- b x\|_{L^{\infty}(S_{\phi}(0,\bar  h))} + \|b\|
 \leq (\theta^{k+1})^{-\frac{1+\alpha}{2}}\|u- b_k x\|_{L^{\infty}(S_{\theta^k})} + \|b_k\|
 \leq C \theta^{-\frac{1+\alpha}{2}}.
  \end{align*}
It remains to show the claim and we  prove it by induction. Let us fix $k\geq 2$ such that 
\begin{equation}\label{InducAss}
\sup_{1\leq m\leq k}N_{\phi, f, q, 2\theta^m}(0)\leq C_2 \delta.
 \end{equation}
 Thanks to \eqref{boundary-u} and $\alpha<\gamma$, (i) and (ii)  clearly hold for $m=1$. Suppose (i) and (ii) hold up to $m\in \{1,\dots, k-1\}$. We 
prove them for $m+1$. As a consequence of \eqref{InducAss}, we have
$$N_{\phi, f, q, 2\theta^{m+1}}(0)\leq C_2 \delta.$$ 
Let $ h := \theta^m$. We define 
the rescaled domain $\Omega_h$ and function $\phi_h$ as in \eqref{phiomega-h}.  
For $x\in \Omega_h$, let
$$v(x):= \frac{(u-l_m)(h^{1/2} A^{-1}_h x)}{h^{\frac{1+\alpha}{2}}},~f_h(x):
= h^{\frac{1-\alpha}{2}}f(h^{1/2} A^{-1}_h x),$$
and 
$$ \Phi_h(x) = (\Phi_h^{ij}(x))= (\det D^2\phi_h(x))\left(D^2\phi_h(x)\right)^{-1}.$$
Then, by (\ref{trans1}), $\Phi_h^{ij}v_{ij}=f_h~\text{in}~S_{\phi_h}(0,1)$ with
$\|v\|_{L^{\infty}(S_{\phi_h}(0,1))}\leq 1$
and
\begin{equation}\label{eq:NvsN}
N_{\phi_h, f_h, q, 2\theta}(0)= N_{\phi, f, q, 2\theta h}(0)=N_{\phi, f, q, 2\theta^{m+1}}(0)\leq C_2 \delta.
\end{equation}
The first inequality in \eqref{eq:NvsN}  follows from \eqref{Amap}--\eqref{phiomega-h} and
\begin{equation*}
 r^{\frac{1-\alpha}{2}}\left(\frac{1}{|S_{\phi_h}(0, r)|}\int_{S_{\phi_h}(0, r)}{|f_h|^q ~dx}\right)^{\frac{1}{q}}
 = (rh)^{\frac{1-\alpha}{2}}\left(\frac{1}{|S_\phi(0, hr)|}\int_{S_\phi(0, hr)}{|f|^q ~dx}\right)^{\frac{1}{q}}~\text{for all }r>0.
\end{equation*}
Define $\varphi_h$ as follows: $\varphi_h =0$ on $\p S_{\phi_h}(0, 2\theta)\cap\p\Omega_h$ and $\varphi_h=
                  v$ on $\p S_{\phi_h}(0, 2\theta) \cap \Omega_h.$
Let $w$ solve
\begin{equation*}
 \left\{
 \begin{alignedat}{2}
   \Phi_h^{ij}w_{ij} ~& = 0 ~&&\text{in} ~ S_{\phi_h}(0, 2\theta), \\\
w &= \varphi_h~&&\text{on}~\p S_{\phi_h}(0, 2\theta).
 \end{alignedat} 
  \right.
\end{equation*} 
By the maximum principle, we have
$$\|w\|_{L^{\infty}(S_{\phi_h}(0, 2\theta))}\leq \|v\|_{L^{\infty}(S_{\phi_h}(0, 2\theta))}\leq 1.$$
Let 
$\bar{l}(x):= \bar{b}x_n$ where $\bar{b}:=\partial_n w(0).$
Then Theorem \ref{LS-gradient} gives
\begin{equation}\abs{\bar{b}}\leq C_0 \|w\|_{L^{\infty}(S_{\phi_h}(0, 2\theta))}\leq C_0
 \label{b-bar}
\end{equation}
and
\begin{eqnarray}\|w-\bar{l}\|_{L^{\infty}(S_{\phi_h}(0, \theta))} 
\leq C_0  (\theta^{\frac{1}{2}})^{1+\alpha_0}\|w\|_{L^{\infty}(S_{\phi_h}(0, 2\theta))}
\leq C_0 (\theta^{\frac{1}{2}})^{1+\alpha_0}
\leq \frac{1}{2}(\theta^{\frac{1}{2}})^{1+\alpha},
\label{alpha0}
\end{eqnarray}
provided that $\theta$ is universally small. Given this, by reducing $\theta$ further if necessary, we show that
\begin{equation}\|w-v\|_{L^{\infty}(S_{\phi_h}(0, 2\theta))} \leq \frac{1}{2}(\theta^{\frac{1}{2}})^{1+\alpha}.
 \label{w-eq}
\end{equation}
Combining this with (\ref{alpha0}), we obtain
\begin{equation}\label{v-barl}
\|v-\bar{l}\|_{L^{\infty}(S_{\phi_h}(0, \theta))}\leq (\theta^{\frac{1}{2}})^{1+\alpha}.
\end{equation}
Now, let
$$l_{m+1}(x):= l_m(x) + (h^{1/2})^{1+\alpha} \bar{l}(h^{-1/2}A_h x).$$
Then, from the definition of $v$ and $\bar l$, and \eqref{v-barl}, we find 
$$\|u-l_{m+1}\|_{L^{\infty}(S_{\theta^{m+1}})} = (h^{1/2})^{1+\alpha}\|v- \bar{l}\|_{L^{\infty}(S_{\phi_h}(0,\theta))}\leq 
 (h^{1/2})^{1+\alpha}  (\theta^{1/2})^{1+\alpha}=  (\theta^{\frac{m+1}{2}})^{1+\alpha},$$
 proving (i). On the other hand, by (\ref{Amap}), we have
 $$l_{m+1} (x) = b_{m+1}x_n~\text{with}~b_{m+1}:= b_m + (h^{1/2})^{1+\alpha} h^{-1/2} \bar{b} = b_m + h^{\alpha/2}\bar{b}.$$
 Therefore, the claim is established since (ii) follows from (\ref{b-bar}) and
 $$\abs{b_{m+1}-b_m}= h^{\alpha/2}\abs{\bar{b}}= \theta^{m\alpha/2} \abs{\bar{b}}.$$
 It remains to prove (\ref{w-eq}). We will apply Lemma \ref{lm:glob-maximum-prin} to $w-v$ which solves
 \begin{equation*}
 \left\{
 \begin{alignedat}{2}
   \Phi_h^{ij}(w-v)_{ij} ~& = -f_h ~&&\text{in} ~ S_{\phi_h}(0, 2\theta), \\\
w-v &= \varphi_h-v~&&\text{on}~\p S_{\phi_h}(0, 2\theta).
 \end{alignedat} 
  \right.
\end{equation*} 
By this lemma and the way $\varphi_h$ is defined, we have
\begin{equation*}\|w-v\|_{L^{\infty}(S_{\phi_h}(0, 2\theta))} \leq \|v\|_{L^{\infty}(\p S_{\phi_h}(0, 2\theta)\cap \p\Omega_h)}
+ C_{\ast}|S_{\phi_h}(0, 2\theta)|^{\frac{2}{n}-\frac{1}{q}} \|f_h\|_{L^{q}(S_{\phi_h}(0, 2\theta))}=: (\mathrm{I}) + (\mathrm{II}),
\end{equation*}
where $C_{\ast}$ depends only on $n,\lambda,\Lambda,\rho$ and $q$.

We estimate $(\mathrm{I})$ as in the proof of \cite[Theorem 1.1]{LS2} and find that if $\theta$ is small then

$$(\mathrm{I}) \leq \frac{1}{4} (\theta^{1/2})^{1+\alpha}.$$
To estimate $(\mathrm{II})$, we recall $N_{\phi_h, f_h, q, 2\theta}(0)\leq C_2 \delta = C_2 (\theta^{1/2})^{1+\alpha}$, 
and note that
$$\|f_h\|_{L^{q}(S_{\phi_h}(0, 2\theta))}\leq N_{\phi_h, f_h, q, 2\theta}(0) (2\theta)^{-\frac{1-\alpha}{2}} 
|S_{\phi_h}(0, 2\theta)|^{\frac{1}{q}}\leq C_2 \delta 
 (2\theta)^{-\frac{1-\alpha}{2}} |S_{\phi_h}(0, 2\theta)|^{\frac{1}{q}}.$$
We therefore obtain from the volume estimates (\ref{vol_ht})
\begin{eqnarray*}
 (\mathrm{II}) &=&C_{\ast}|S_{\phi_h}(0, 2\theta)|^{\frac{2}{n}-\frac{1}{q}} \|f_h\|_{L^{q}(S_{\phi_h}(0, 2\theta))}\leq
 C_{\ast} C_2|S_{\phi_h}(0, 2\theta)|^{\frac{2}{n}}(2\theta)^{-\frac{1-\alpha}{2}}\delta
 \\&\leq& C_\ast C_2 C_1^{2/n} (2\theta)^{\frac{1+\alpha}{2}}\delta
  \leq \frac{1}{4} (\theta^{1/2})^{1+\alpha} 
\end{eqnarray*}
if $\theta$ is small. It follows that
$$\|w-v\|_{L^{\infty}(S_{\phi_h}(0, 2\theta))} \leq (\mathrm{I}) + (\mathrm{II})\leq \frac{1}{2}(\theta^{\frac{1}{2}})^{1+\alpha},$$
proving \eqref{w-eq}. The proof of Theorem~\ref{h-bdr-gradient} is complete.
\end{proof}
\section{Proof of the global $W^{1,p}$ and H\"older estimates} 
\label{sec:globalw1p}
In this section, we prove the main result of the paper (Theorem~\ref{global-reg}) regarding global $W^{1,p}$ estimates
for solutions to \eqref{LMA-eq}. We also prove the global H\"older estimates in Theorem~\ref{global-H}.
\subsection{Global $W^{1,p}$ estimates} Before giving the proof of Theorem \ref{global-reg}, we indicate its overall structure.
First, we bound the solution using the global maximum principle in Lemma~\ref{lm:glob-maximum-prin2}. Then, using a consequence of the boundary Localization Theorem
for the Monge-Amp\`ere equations \cite{S1, S2}, 
we combine the pointwise $C^{1,\alpha}$ estimates in the interior
and at the boundary in Theorems~\ref{HolderDv} and \ref{h-bdr-gradient} to bound the gradient by the function $N$ defined in (\ref{Ndef}).
The rest of the proof of Theorem~\ref{global-reg} is similar to that of Theorem~\ref{Interior-Du}. Here, we use the
global strong type estimate for the maximal function $\mathcal{M}$ in Theorem~\ref{strongtype}
and the volume growth of sections of $\phi$.  Notice that by \cite[Corollary 2.4]{LN1}, there exist constants $c_{\ast}, C_{1}, C_{2}$ depending only on $\rho, \lambda, \Lambda$ and $n$ 
such that for any section $S_{\phi}(x, t)$ with $x\in\overline{\Omega}$ and $t\leq c_{\ast}$, we have
\begin{equation}
\label{big_sec_vol}
C_{1} t^{n/2}\leq |S_{\phi}(x, t)|\leq C_{2} t^{n/2}.
\end{equation}
\begin{proof}[Proof of Theorem \ref{global-reg}] We extend $\varphi$ to a $C^{1,\gamma}(\overline{\Omega})$ function in $\overline{\Omega}$.
By multiplying $u$ by a suitable constant, we can assume that $$\|f\|_{L^{q}(\Omega)}
+ \|\varphi\|_{C^{1,\gamma}(\overline{\Omega})}\leq 1.$$
By the global maximum principle in Lemma~\ref{lm:glob-maximum-prin2}, we have
 \begin{equation}
  \label{u-max}
  \|u\|_{L^{\infty}(\Omega)} \leq C \left(\|f\|_{L^{q}(\Omega)} + \|\varphi\|_{L^{\infty}(\Omega)}\right)\leq C
 \end{equation}
for some $C$ depending on $n, q, \rho, \lambda$, and $\Lambda$. It remains to show that for all $p<\frac{nq}{n-q}$, we have
\begin{equation}
\label{DuLp}
 \|Du\|_{L^{p}(\Omega)} \leq C(n, p, q, \gamma,\rho, \lambda,\Lambda).
\end{equation}

By using Theorem~\ref{Interior-Du} and restricting our estimates in small balls of definite size around $\p\Omega$, we can assume throughout
 that $1-\theta\leq g\leq 1+ \theta$ where $\theta$ is the smallest of the two $\theta$'s in Theorems~\ref{HolderDv} and \ref{h-bdr-gradient}.
 
 Let $y\in \Omega $ with $r:=\dist (y,\partial\Omega) \le c,$ for $c$ universal ($c\ll\theta$).
Since $\phi$ is $C^{1, 1}$ on the boundary $\p\Omega$, by Caffarelli's strict convexity theorem \cite{C90}, $\phi$ is strictly convex in $\Omega$. This implies 
the existence of the maximal interior section $S_{\phi}(y, h)$ of $\phi$ centered at $y$ with 
$h:=\sup\{t\,| S_{\phi}(y,t)\subset \Omega\}>0.$
By \cite[Proposition 3.2]{LS1} applied at the point $x_0\in \p S_{\phi}(y,h) \cap \p \Omega,$ we have
 \begin{equation} h^{1/2} \sim r,
\label{hr}
\end{equation}
and $ S_{\phi}(y,h)$ is equivalent to an ellipsoid $E$, that is,
$cE \subset  S_{\phi}(y,h)-y \subset CE,$
where
\begin{equation}E :=h^{1/2}A_{h}^{-1}B_1(0), \quad \mbox{with} \quad \|A_{h}\|, \|A_{ h}^{-1} \| \le C |\log  h|;\, \, \det A_{h}=1.
\label{eh}
\end{equation}
Moreover, by \cite[Theorem 2.1]{LN1}, we have the engulfing property of sections of $\phi$. That is, 
there exists $\theta_{\ast}>0$ depending only 
on $\rho, \lambda, \Lambda$ and $n$ such that if $y\in S_\phi(x, t)$ with $x\in\overline{\Omega}$ and $t>0$, then $S_\phi(x, t)\subset S_\phi(y, \theta_{\ast}t).$
Hence,  for any $z\in S_\phi(y, h)$ the following inclusions hold: 
\begin{equation}z\in S_{\phi}(y, h)\subset S_\phi(x_0,\theta_\ast h) \subset S_\phi(x_0,2\theta^{-1} t) \subset S_\phi(z,2\theta_\ast \theta^{-1}t)\quad
\text{for all }t\geq\theta_{\ast}h.
 \label{yx0h}
\end{equation}

Let $q'$ be such that $\frac{n}{2}<q'<q$.
By Theorem \ref{h-bdr-gradient} applied to the original function $u$ in $S_\phi(x_0,\theta_\ast h)$, we can find $b\in\R^n$  and a universal constant $C$ such that
\begin{multline}
 (\theta_\ast h)^{-\frac{1+\alpha}{2}}\|u(x)- u(x_0)- b(x-x_0)\|_{L^{\infty}(S_\phi(x_0,\theta_\ast h))} +\|b\|\\ \leq 
 C\Big[\|u\|_{L^{\infty}(\Omega)}
+\|\varphi\|_{C^{1,\gamma}(\overline{\Omega})} +\sup_{\theta_\ast h \leq t \leq \theta^2} N_{\phi, f, q', 2\theta^{-1}t}(x_0)
\Big],
\label{umax_res}
\end{multline}
where in the definition of $N_{\phi, f, q', 2\theta^{-1}t}(x_0)$ in (\ref{Nr1def}), $\alpha\in (0,1)$ is the exponent in 
Theorem \ref{h-bdr-gradient}.

We now use (\ref{eh}) to rescale our equation. The rescaling $\tilde \phi$ of $\phi$ 
$$\tilde \phi(\tilde x):=\frac {1}{ h} \left[\phi(y+ h^{1/2}A_{h}^{-1}\tilde x)-\phi (y) -D \phi(y) (h^{1/2}A_{h}^{-1}\tilde x)\right]$$
satisfies
$$\det D^2\tilde \phi(\tilde x)=\tilde g(\tilde x):=g(y+ h^{1/2}A_{h}^{-1}\tilde x)\in [1-\theta,1+\theta],  $$
and
\begin{equation}
\label{normalsect}
B_{c}(0) \subset  S_{\tilde\phi}(0, 1) \subset B_{C}(0), \quad \quad  S_{\tilde\phi}(0, 1)= h^{-1/2} A_{ h}\big(S_\phi(y, h)- y\big),
\end{equation}
where we recall that $ S_{\tilde\phi}(0, 1)$ represents the section of $\tilde \phi$ at the origin with height 1. We denote $\tilde S_{t}= S_{\tilde\phi}(0, t).$
We define also the rescaling $\tilde u$ for $u$
$$\tilde u(\tilde x):= h^{-1/2}\big[u(x)- u(x_{0})-b(x-x_0)\big],\quad \tilde x\in \tilde S_{1}, \quad x=T\tilde x:=y+ h^{1/2}A_{h}^{-1}\tilde x.$$
Let $\tilde\Phi= (\tilde \Phi^{ij})_{1\leq i, j\leq n}$ be the cofactor matrix of $D^2\tilde\phi$. Then, by (\ref{trans1}), $\tilde u$ solves
$$\tilde \Phi^{ij} \tilde u_{ij} = \tilde f(\tilde x):= h^{1/2} f(T\tilde x).$$
From (\ref{umax_res}), (\ref{u-max}) and (\ref{hr}), we have
\begin{eqnarray}\|\tilde u\|_{L^{\infty}(\tilde S_{1})}& \leq& C h^{-1/2} (\theta_\ast h)^{\frac {1+ \alpha}{2}} \Big[\|u\|_{L^{\infty}(\Omega)}
+\|\varphi\|_{C^{1,\gamma}(\overline{\Omega})}  + \sup_{\theta_\ast h \leq t \leq \theta^2} N_{\phi, f, q', 2\theta^{-1}t}(x_0)
\Big]\nonumber \\&\leq&
Cr^{\alpha} \Big[1+  \sup_{\theta_\ast h \leq t \leq \theta^2} N_{\phi, f, q', 2\theta^{-1}t}(x_0) \Big].
\label{umax_res2}
\end{eqnarray}
Now, in the definition of $N$ in (\ref{Ndef}), we let $\alpha\in (0,1)$ be the exponent in Theorem \ref{h-bdr-gradient} and $r_0=2\theta_\ast \theta$.
Apply Theorem~\ref{HolderDv} to $\tilde u $ and arguing as in (\ref{iteration-consequence}), we obtain
$$\abs{D\tilde u (\tilde z)}\leq C\Big[\|\tilde u \|_{L^{\infty}(\tilde S_{1})} + N_{\tilde \phi, \tilde f, q'}(\tilde z)\Big]\quad\text{for a.e. } \tilde z
\in \tilde S_{1/2}.$$
Note that, by (\ref{eh}) and (\ref{hr}),
\begin{equation}
 N_{\tilde \phi, \tilde f, q'}(\tilde z)\leq h^{\frac{\alpha}{2}} N_{\phi, f, q'}(z) \leq C r^{\alpha} N_{\phi, f, q'}(z)\quad \mbox{with }\, z=T\tilde z.
 \label{scaled-lp}
 \end{equation}
It is easy to see from the definitions of $ N_{\phi, f, q', 2\theta^{-1} t}(x_0)$ and $ N_{\phi, f, q'}(z)$, (\ref{yx0h}) and the volume estimates in (\ref{big_sec_vol}) that
\begin{equation}N_{\phi, f, q', 2\theta^{-1}t}(x_0)\leq CN_{\phi, f, q', 2\theta_{\ast}
\theta^{-1}t}(z) \leq  CN_{\phi, f, q'}(z) \quad \text{for all }t\in [ \theta_\ast h , \theta^2].
\label{Ncomp}
\end{equation}
Hence, using (\ref{umax_res2}) and (\ref{scaled-lp}), we get
\begin{align*}
\abs{D\tilde u (\tilde z)}\leq C r^{\alpha}\Big[1 + N_{\phi, f, q'}(z) + \sup_{\theta_\ast h \leq t \leq \theta^2} N_{\phi, f, q', 2\theta^{-1}t}(x_0)\Big] \leq 
C r^{\alpha}\big[1 + N_{\phi, f, q'}(z) \big]
\end{align*}
for a.e. $ \tilde z=T^{-1} z\in \tilde S_{1/2}$.
 Rescaling back, using 
$$\tilde z= h^{-1/2}A_{ h}(z-y),\quad D\tilde u (\tilde z)= (A_h^{-1})^t(Du(z)-b)\, \text{ and }\,  h^{1/2}\sim r,$$
together with (\ref{umax_res}) and (\ref{Ncomp}),
we find for all $z \in   S_{\phi}(y,h/2)$ that
\begin{eqnarray*}|Du(z)| =|A_{h}^tD\tilde u (\tilde z) + b|  \leq C\abs{\log h}r^{\alpha} \big[1 + N_{\phi, f, q'}(z) \big] + C
\big[1 + N_{\phi, f, q'}(z) \big]
\leq C\big[ 1 + N_{\phi, f, q'}(z) \big].
\label{oscv}
\end{eqnarray*}
In particular, we obtain the following gradient estimate for a.e. $y\in\Omega$ with $dist(y,\p\Omega)=r\leq c, $
$$|Du(y)| \leq C\big[ 1 + N_{\phi, f, q'}(y)  \big].$$
This is a global version of \eqref{iteration-consequence}.
Now, we argue as in the proof of Theorem~\ref{Interior-Du} and using a global version of strong type estimate for the maximal function in Theorem \ref{strongtype} 
and the volume growth of sections in (\ref{big_sec_vol}) to conclude
the proof of Theorem \ref{global-reg}.
\end{proof}

\subsection{ Global H\"older  Estimates} 
\begin{proof}[Proof of Theorem \ref{global-H}]
The proof of the global H\"older 
estimates in this theorem is similar to the proofs  of \cite[Theorem~1.4]{Le13} and \cite[Theorem 4.1]{LN2}. 
It combines the boundary H\"older estimates in Proposition~\ref{local-H} and the interior H\"older continuity estimates  in Corollary \ref{cor:imterior-holder} using
Savin's Localization Theorem
\cite{S1, S2}. Thus we omit the details and only present the proof of Proposition~\ref{local-H} below.
\end{proof}

\begin{proposition}  \label{local-H}
Let $\phi$ and $u$ be as in Theorem \ref{global-H}. Then, there exist $\delta, C$ depending only on $\lambda, \Lambda, n,  \alpha, \rho$ and $q$
such that, for any $x_{0}\in\partial\Omega\cap B_{\rho/2}(0)$, we have
$$|u(x)-u(x_{0})|\leq C|x-x_{0}|^{\frac{\alpha_0}{\alpha_0 +3n}} \Big(
\|u\|_{L^{\infty}(\Omega\cap B_{\rho}(0))} + \|\varphi\|_{C^\alpha(\partial\Omega\cap B_{\rho}(0))}  
+ \|f\|_{L^{q}(\Omega\cap B_{\rho}(0))} \Big)~\text{for all } x\in \Omega\cap B_{\delta}(x_{0}), $$
 where
 $$\alpha_0 :=\min\big\{\alpha, \frac{3}{8}(2-\frac{n}{q})\big\}.$$
\end{proposition}

The proof of Proposition~\ref{local-H} relies on an extension of Lemma \ref{lm:glob-maximum-prin} and 
a construction of suitable barriers. 

In what follows, we assume $\phi$ and $\Omega$ satisfy the assumptions 
in the proposition. 
We also assume for simplicity that
 $\phi(0)=0\, \mbox{ and } \, \nabla \phi(0)=0.$ Furthermore, we abbreviate $B_r(0)$ by $B_r$ for $r>0$.
 
We now recall the  following construction of supersolution in \cite{LN2}.
 \begin{lemma}(\cite[Lemma 4.4]{LN2})
Given $\delta$ universally small ($\delta\leq \rho$), define
$$\tilde \delta:= \frac{\delta^3}{2}\quad \mbox{and} \quad M_{\delta}:= \frac{2^{n-1} \Lambda^n}{\lambda^{n-1}}\frac{1}{\delta^{3n-3}}
\equiv \frac{\Lambda^{n}}{(\lambda \tilde \delta)^{n-1}}.$$
Then the function
$$w_{\delta}(x', x_n):= M_{\delta} x_{n} +  \phi - \tilde \delta |x'|^2 - \frac{\Lambda^n}{( \lambda \tilde \delta)^{n-1}} 
 x_{n}^2\quad\mbox{for}\quad (x', x_n)\in \overline{\Omega}$$ satisfies 
 $$\Phi^{ij}(w_{\delta})_{ij} \leq - n\Lambda\quad\mbox{in}\quad \Omega,$$
and
$$w_{\delta} \ge 0 ~\text{ on }~ \p(\Omega \cap B_\delta), \quad w_{\delta} \ge \frac{\delta^3}{2} ~\text{ on }~  \Omega\cap \p B_{\delta}.$$
\label{key-lem_sup}
\end{lemma}
The next result is an extension of Lemma~\ref{lm:glob-maximum-prin} where sections are now replaced by balls. 
 \begin{lemma}\label{global-ball} 
Let $A= \Omega\cap B_\delta(0)$ where $\delta\leq c$ with $c= c(n,\lambda,\Lambda, \rho)$ being universally small. 
Assume that $f\in L^q(A)$ for some $q>n/2$  and  $u\in W^{2,n}_{loc}(A)\cap C(\overline{A})$ satisfies
\[
\mathcal{L}_{\phi} u\leq f\quad \mbox{almost everywhere in}\quad A.
\]
Then  there exists a constant $C>0$ depending only on  $n$, $\lambda$, $\Lambda$, $\rho$ and  $q$  such that
\[
\sup_{A}{u} \leq \sup_{\partial A}{u^+} + C |A|^{\frac{3}{4}(\frac{2}{n} -\frac{1}{q})}   \|f\|_{L^q(A)}. 
\]
\end{lemma}
\begin{proof}
Let $G_A(\cdot, y)$ be the Green's function of $\mathcal{L}_{\phi}$ in $A$ with pole $y\in A$. As in the proof of Lemma~\ref{lm:glob-maximum-prin}, it suffices to prove that 
\begin{equation}
 \label{glGA}
  \|G_A(x, \cdot)\|_{L^{q'}(A)}\leq C|A|^{\frac{3}{4}(\frac{2}{n}-\frac{1}{q})}\quad \text{for all~}x\in A.
 \end{equation}
 Note that from (\ref{loc-k}) and (\ref{Amap}) we have for $h\leq c$
$$ \overline \Omega \cap B^{+}_{ch^{1/2}/\abs{\log h}}\subset S_{\phi}(0, h) 
\subset \overline{\Omega} \cap B^{+}_{C h^{1/2} \abs{\log h}}.$$
Hence for $|x|\leq \delta\leq c$, we deduce from the first inclusion that
\begin{equation}\label{AVeq} A=\Omega\cap B_\delta(0)\subset S_\phi(0, \delta^{3/2}):=V.
\end{equation}
Arguing as in  \eqref{Gest1}, \eqref{Gphih}, we find that 
\begin{equation}
 \label{glGV2}
  \|G_V(x, \cdot)\|_{L^{q'}(V)}\leq C|V|^{\frac{2}{n}-\frac{1}{q}}\quad \text{for all~}x\in V.
 \end{equation}
Using the volume estimate for sections in (\ref{vol_ht}), we find that
$$|V|\leq C \delta^{\frac{3n}{4}}\leq C|A|^{\frac{3}{4}}.$$
This together with (\ref{glGV2}) and (\ref{AVeq}) implies (\ref{glGA}).
\end{proof}

\begin{proof}[Proof of Proposition \ref{local-H}] Our proof follows closely the proof of Proposition 2.1 in \cite{Le13}. We include here the details for reader's convenience. Since
$$\|\varphi\|_{C^{\alpha_0}(\partial\Omega\cap B_{\rho})} \leq C(\alpha_0, \alpha,\rho) \|\varphi\|_{C^\alpha(\partial\Omega\cap B_{\rho})},$$
it suffices to show that
$$|u(x)-u(x_{0})|\leq C|x-x_{0}|^{\frac{\alpha_0}{\alpha_0 +3n}} \Big(
\|u\|_{L^{\infty}(\Omega\cap B_{\rho})} + \|\varphi\|_{C^{\alpha_0}(\partial\Omega\cap B_{\rho})}  + 
\|f\|_{L^{q}(\Omega\cap B_{\rho})} \Big)~\text{for all }x\in \Omega\cap B_{\delta}(x_{0}).$$
We can suppose  that  $K:= \|u\|_{L^{\infty}(\Omega\cap B_{\rho})} + \|\varphi\|_{C^{\alpha_0}(\partial\Omega\cap B_{\rho})}  + \|f\|_{L^{q}(\Omega\cap B_{\rho})}$ is finite. By  working with the function $v:= u/K$ instead of $u$, we can assume in addition that
\[
\|u\|_{L^{\infty}(\Omega\cap B_{\rho})} + \|\varphi\|_{C^{\alpha_0}(\partial\Omega\cap B_{\rho})}  + \|f\|_{L^{q}(\Omega\cap B_{\rho})}\leq 1
\]
and need to show that the inequality
\begin{equation}\label{equivalent-conclusion}
|u(x)-u(x_{0})|\leq C|x-x_{0}|^{\frac{\alpha_0}{\alpha_0 +3n}}\quad\text{for all }x\in \Omega\cap B_{\delta}(x_{0})
\end{equation}
holds for all $x_0\in \Omega\cap B_{\rho/2}$,  where  $\delta$ and $C$ depend only on $\lambda, \Lambda, n,  \alpha$, $\rho$ and $q$. 

We prove \eqref{equivalent-conclusion} for $x_{0} =0$. However, our arguments apply to all points $x_0\in \Omega\cap B_{\rho/2}$ with
obvious modifications.
For any $\varepsilon \in (0,1)$, we consider the functions
$$h_{\pm}(x) := u(x)- u(0)\pm \e\pm \frac{6}{\delta_2^3} w_{\delta_2}$$
in the region
 $$A:= \Omega\cap B_{\delta_{2}}(0),$$ where $\delta_{2}$ is small to be chosen later and the function $w_{\delta_2}$ is as in Lemma~\ref{key-lem_sup}.
We remark that $w_{\delta_2}\geq 0$ in $A$ by the maximum principle. Observe that if $x\in\partial\Omega$
with $|x|\leq \delta_{1}(\e):= \e^{1/\alpha_0}$ then, 
\begin{equation}
\label{bdr-ineq}|u(x)-u(0)| =|\varphi(x)-\varphi(0)| \leq |x|^{\alpha_0} \leq \e.
\end{equation}
On the other hand, if $x\in \Omega \cap \p B_{\delta_2}$ then from Lemma~\ref{key-lem_sup}, we obtain
$\frac{6}{\delta_2^3} w_{\delta_2}(x)\geq 3.$
It follows that, if we choose $\delta_{2}\leq \delta_{1}$ then from \eqref{bdr-ineq} and $|u(x)-u(0)\pm \e|\leq 3$, we get
$$h_{-}\leq 0, \, h_{+}\geq 0~\text{ on }~\partial A.$$
Also from Lemma~\ref{key-lem_sup}, we have 
$$-\mathcal{L}_{\phi} h_{+}\leq f, \,  -\mathcal{L}_{\phi} h_{-}\geq f~\text{ in }~A.$$
Here we recall that $\mathcal{L}_{\phi}=-\Phi^{ij} \p_{ij}.$
Hence Lemma \ref{global-ball} applied in $A$ gives the following estimates
\begin{equation}\label{h-from-above}
h_{-}\leq  C_1 |A|^{\frac{3}{4}(\frac{2}{n}-\frac{1}{q})} \|f\|_{L^{q}(A)}\leq C_{1}\delta_{2}^{\frac{3}{4}(2-\frac{n}{q})}~\text{ in }~ A
\end{equation}
and 
\begin{equation}\label{h-from-below}
h_{+}\geq - C_{1}|A|^{\frac{3}{4}(\frac{2}{n}-\frac{1}{q})} \|f\|_{L^{q}(A)}\geq  -C_1 \delta_{2}^{\frac{3}{4}(2-\frac{n}{q})}~\text{ in }~ A
\end{equation}
where $C_1>1$ depends only on $n,\lambda,\Lambda, \rho$ and $q$.
By restricting $\e\leq C_{1}^{-1}(\leq 1)$, we can assume that
$$\delta_{1}^{\frac{3}{4}(2-\frac{n}{q})} = \e^{{(2-\frac{n}{q}})\frac{3}{
4\alpha_0}}\leq \e^2\leq \frac{\e}{C_{1}}.$$
Then, for $\delta_{2}\leq \delta_{1}$, we have $C_{1}\delta_{2}^{\frac{3}{4}(2-\frac{n}{q})}\leq \e$ and thus, for all $x\in A$, we obtain from \eqref{h-from-above} and \eqref{h-from-below} that
$$|u(x)-u(0)|\leq 2\e + \frac{6}{\delta_2^3} w_{\delta_2}(x).$$
Note that, by construction and the boundary estimate for the function $\phi$, we have in $A$
$$w_{\delta_2}(x)\leq M_{\delta_2} x_n + \phi(x)\leq M_{\delta_2}\abs{x} + C\abs{x}^2\abs{\log \abs{x}}^2\leq 2M_{\delta_2}\abs{x}.$$

Therefore, choosing $\delta_{2}= \delta_{1}$ and recalling the choice of $M_{\delta_2}$, we get
\begin{equation}
\label{op-ineq}
|u(x)-u(0)|\leq 2\e + \frac{12 M_{\delta_2}}{\delta_2^3}\abs{x}= 2\e + \frac{C_{2}(n,\lambda, \Lambda)}{\delta_{2}^{3n}}|x| = 2\e + C_{2} 
\e^{-\frac{3n}{\alpha_0}}|x|
\end{equation}
for all $x,\e$ satisfying the following conditions
\begin{equation*}
|x|\leq \delta_{1}(\e):= \e^{1/\alpha_0},\quad \e\leq C_{1}^{-1}.
\end{equation*}
Finally, let us choose $\e = |x|^{\frac{\alpha_0}{\alpha_0 + 3n}}.$
It satisfies the above conditions if
$|x|\leq C_{1}^{-\frac{\alpha_0 +3n}{\alpha_0}}=: \delta.$
Then, by \eqref{op-ineq}, we have 
$|u(x)-u(0)| \leq (2 + C_2) |x|^{\frac{\alpha_0}{\alpha_0 + 3n}}\,$ for all $\, x\in \Omega\cap B_{\delta}(0)$.
\end{proof}

\bibliographystyle{plain}

\begin{thebibliography}{10}
\bibitem{Br} Brenier, Y. Polar factorization and monotone rearrangement of vector-valued functions. {\it Comm. Pure Appl. Math.} {\bf 44} (1991), no. 4, 375-417.
\bibitem{C89} Caffarelli, L. A. Interior a priori estimates for solutions of fully nonlinear equations. {\it Ann. of Math.} {\bf 130} (1989),  no. 1, 189--213.

\bibitem{C90} Caffarelli, L.A. A localization property of viscosity
    solutions to the Monge-Amp\`ere equation and their strict convexity.
    {\it Ann. of Math.} {\bf 131} (1990), no. 1, 129--134.
\bibitem{C} Caffarelli, L. A. Interior $W^{2,p}$
estimates for solutions to the Monge-Amp\`ere equation.
    {\it Ann. of Math.} {\bf 131} (1990), no. 1, 135--150.

\bibitem{CC} Cafarelli, L.A.; Cabr\'e, X. {\em Fully nonlinear elliptic
equations.} American Mathematical Society.  Colloquium Publications, 43. American Mathematical Society, Providence, RI, 
1995. 
 
\bibitem{CG97} Caffarelli, L. A.; 
Guti\'errez, C. E. Properties of the Solutions of the
Linearized {M}onge--{A}mp\`ere equation. {\it Amer. J. Math.} {\bf
119}(1997), no. 2, 423--465.
\bibitem{CNP91} Cullen, M. J. P.; Norbury, J.; Purser, R. J. Generalized Lagrangian solutions for atmospheric and oceanic flows. {\it SIAM J. Appl. Anal.} {\bf 51} (1991), no. 1, 
20--31. 
\bibitem{D05} Donaldson, S. K. Interior estimates for solutions of Abreu's equation.  {\it Collect. Math.}  {\bf 56}  (2005),  no. 2, 103--142.
\bibitem{Es93} Escauriaza, L.
$W^{2,n}$ a priori estimates for solutions to fully non-linear equations. {\it Indiana Univ. Math. J.} {\bf 42} (1993), no. 2, 413--423.
\bibitem{FS} Fabes, E. B.; Stroock, D. W. The $L^p$-integrability of Green's functions and fundamental solutions for elliptic and parabolic equations. 
{\it Duke Math. J.} {\bf 51} (1984), no. 4, 997--1016.

\bibitem{GiT} Gilbarg, D.; Trudinger, N. S.
{\em Elliptic partial differential equations of second order.}
Springer--Verlag, New York, 2001.

\bibitem{GW} Gr\"uter, M.; Widman, K. O. The Green function for uniformly elliptic equations.
 {\it Manuscripta Math.}  {\bf 37} (1982),  no. 3, 303--342.

 \bibitem{G01} Guti\'errez, C. E. {\em The Monge-Amp\`ere Equation.}
  Birkha\"user, Boston, 2001.

 \bibitem{GN1}
Guti\'errez, C. E.; Nguyen, T. 
\newblock  Interior gradient estimates for solutions to the linearized
 Monge-Amp\`ere equation.
 {\it Adv. Math.}  {\bf 228}  (2011),  no. 4, 2034--2070.

\bibitem{GN2}
Guti\'errez, C. E.; Nguyen, T. 
\newblock Interior second derivative estimates for solutions to the linearized
 Monge-Amp\`ere equation.
 {\it Trans. Amer. Math. Soc.}  {\bf 367}  (2015),  no. 7, 4537--4568.
\bibitem{Le13} Le, N. Q. Global second derivative estimates for 
the second boundary value problem of the prescribed affine mean curvature and Abreu's equations.
 {\it Int. Math. Res. Not. IMRN} (2013), no. 11, 2421--2438.
 \bibitem{L} Le, N. Q. Remarks on the Green's function of the linearized Monge-Amp\`ere operator. {\it Manuscripta Math.}
 {\bf 149} (2016), no. 1, 45--62.
 
\bibitem{Le15}Le, N. Q. Boundary Harnack inequality for the linearized Monge-Amp\`ere equations and applications, arXiv:1511.01462 [math.AP].

\bibitem{LN1}
  Le, N. Q.;  Nguyen, T. 
  \newblock    Geometric properties of boundary sections of solutions to the
 Monge-Amp\`ere equation and applications.
{\it  J. Funct. Anal.}  {\bf 264}  (2013),  no. 1, 337--361.

\bibitem{LN2}
  Le, N. Q.;  Nguyen, T. 
  \newblock  Global $W^{2,p}$  estimates for solutions to the linearized
 Monge-Amp\`ere equations.
 {\it Math. Ann.}  {\bf 358}  (2014),  no. 3-4, 629--700.
\bibitem{LS1} Le, N. Q.; Savin, O. Boundary regularity for solutions to the linearized Monge-Amp\`ere equations.
 {\it Arch. Ration. Mech. Anal.} {\bf 210} (2013), no. 3, 813--836.
\bibitem{LS2} Le, N. Q.; Savin, O. 
On boundary H\"older gradient estimates for solutions to the linearized Monge-Amp\`ere equations. {\it Proc. Amer. Math. Soc.} {\bf 143} (2015), no. 4, 1605--1615. 
\bibitem{Loe} Loeper, G. A fully nonlinear version of the incompressible euler equations: the semigeostrophic system. {\it SIAM J. Math. Anal.} {\bf 38} (2006), no. 3, 795--823.
\bibitem{S1} Savin, O. A localization property at the boundary for the Monge-Amp\`ere equation. {\it Advances in Geometric Analysis}, 45--68,  
{\bf Adv. Lect. Math. (ALM), 21}, Int. Press, Somerville, MA, 2012.


\bibitem{S2} Savin, O. Pointwise $C^{2,\alpha}$ estimates at the boundary for the Monge-Amp\`ere equation. {\it J. Amer.
Math. Soc.} {\bf 26} (2013), no. 1, 63--99. 

\bibitem{Sw97} \'Swiech, A.
$W\sp {1,p}$-interior estimates for solutions of fully nonlinear, uniformly elliptic equations.
 {\it Adv. Differential Equations.}  {\bf 2} (1997),  no. 6, 1005--1027.
 
\bibitem{TiW08} Tian, G. J.; Wang, X. J.
A class of Sobolev type inequalities. {\it Methods Appl. Anal.} {\bf 15} (2008), no. 2, 263--276.
\bibitem{TW00} Trudinger, N. S.; Wang, X. J. The Bernstein problem for affine maximal hypersurfaces.  {\it Invent. Math.}  {\bf 140}  (2000),  no. 2, 399--422.
\bibitem{TW05}
Trudinger, N.S.; Wang, X.J. The affine plateau problem. {\it J. Amer.
Math. Soc.} {\bf 18}(2005), no. 2, 253--289.
\bibitem{TW08} Trudinger, N.S.; Wang X.J. Boundary regularity for Monge-Amp\`ere and affine maximal surface equations. {\it Ann. of Math.} {\bf 167} (2008), no. 3, 993--1028.
\bibitem{TW082} Trudinger, N. S.; Wang, X. J. The Monge-Amp\`{e}re equation and its 
geometric applications.  {\it Handbook of geometric analysis.} No. 1,  467--524, {\bf Adv. Lect. Math. (ALM), 7}, Int. Press, Somerville, MA, 2008.
\bibitem{W95} Wang, X. J.
Some counterexamples to the regularity of Monge--Amp\`ere
equations. {\it  Proc. Amer. Math. Soc.} {\bf 123} (1995), no. 3, 841--845.
\bibitem{Wang} Wang, X. J. Schauder estimates for elliptic and parabolic equations. 
 {\it Chinese Ann. Math. Ser. B} {\bf 27} (2006), no. 6, 637--642.
\bibitem{Wi} Winter, N. $W\sp {2,p}$ and  
$W\sp {1,p}$-estimates at the boundary for solutions of fully nonlinear, uniformly elliptic equations.
{\it  Z. Anal. Anwend.}  {\bf 28} (2009), no. 2, 129--164.

\end{thebibliography}

\end{document}